\documentclass[a4paper,8pt]{article}
\usepackage[english]{babel}
\usepackage[utf8x]{inputenc}
\usepackage[T1]{fontenc}
\usepackage{caption}

\usepackage[a4paper,top=3cm,bottom=2cm,left=3cm,right=3cm,marginparwidth=1.75cm]{geometry}
\usepackage{subcaption}
\usepackage{graphicx}
\usepackage{siunitx}
     \usepackage{setspace}
\usepackage{multirow}
\usepackage{hhline}
\usepackage{mathrsfs} 
\usepackage{amsmath}
\usepackage{amssymb}
\usepackage{amsthm}
\usepackage{dsfont}

\usepackage{multicol}
\usepackage{nomencl}
\makenomenclature

\renewcommand\nompreamble{\begin{multicols}{2}}
\renewcommand\nompostamble{\end{multicols}}

\usepackage{float}
\usepackage{graphicx,subcaption}
\usepackage{color}
\usepackage{tabu}
 
\usepackage{setspace}
\usepackage{setspace}
\usepackage{amsmath}

\makeatletter

\newcolumntype{R}[1]{>{\raggedleft\arraybackslash}m{#1}} 
\newcolumntype{L}[1]{>{\raggedright\arraybackslash}m{#1}} 
\newcolumntype{C}[1]{>{\centering\arraybackslash}m{#1}}

\renewcommand{\fnum@figure}{\small\textbf{\figurename\thefigure}} 
\makeatother
\usepackage[algoruled,vlined,boxed,longend,english,lined,boxed,commentsnumbered]{algorithm2e}
\numberwithin{equation}{section}%
\theoremstyle{plain}
\newtheorem{theorem}{Theorem}[section]

\theoremstyle{definition}
\newtheorem{definition}{Definition}[section]

\theoremstyle{remark}
\newtheorem{remark}{Remark}[section]

\theoremstyle{remarks}
\newtheorem{remarks}{Remarks}[section]

{%

\begin{enumerate}}%
{\end{enumerate}
}

\def\be{\begin{equation}}
\def\ee{\end{equation}}
\def\ba{\begin{aligned}}
\def\ea{\end{aligned}}
\def\bes{\begin{equation*}}
\def\ees{\end{equation*}}
\def\bc{\begin{cases}}
\def\ec{\end{cases}}
\numberwithin{equation}{section}
\everymath{\displaystyle}
\usepackage{lineno}

\usepackage{diagbox}

\title{\textbf{Determination of the flux terms in a time fractional viscoelastic equation}}

\author{Mohamed BenSalah \thanks{ Mohamed BenSalah,\hfil\break
Department of Mathematics \newline
ISSAT of Sousse, University of Sousse, Rue Tahar Ben Achour, Sousse 4003, Tunisia,\hfil\break
E-Mail : mohamed.bensalah@fsm.rnu.tn} \, ,  Salih Tatar \thanks{ Salih Tatar, Corresponding Author \hfil\break
Department of Mathematics $\&$  Computer Science, College of Science and General Studies,\newline
Alfaisal University, Riyadh, KSA,\hfil\break
E-Mail : statar@alfaisal.edu} \, , S\"{u}leyman Ulusoy \thanks{ S\"{u}leyman Ulusoy,\hfil\break
Department of Mathematics and Physics, Faculty of Arts and Sciences,\newline
American  University of Ras Al Khaimah, UAE\hfil\break
E-mail : suleyman.ulusoy@aurak.ac.ae} \, and   Masahiro Yamamoto \thanks{ Masahiro Yamamoto,\hfil\break
Department of Mathematical Sciences, \newline
 The University of Tokyo 3-8-1, Komaba, Meguro, Tokyo 153, Japan,\hfil\break
E-Mail : myama@ms.u-tokyo.ac.jp} \ } 
\begin{document}
\maketitle
\begin{abstract}
In this paper, we study the flux identification problem for a nonlinear time-fractional viscoelastic equation with a general source function based on the boundary measurements. We prove that the direct problem is well-posed, i.e., the solution exists, unique and depends continuously on the heat flux. Then the Fr\'echet differentiability of the cost functional is proved. The Conjugate Gradient Algorithm, based on the gradient formula for the cost functional, is proposed for numerical solution of the inverse flux problem. The numerical examples, both with noise-free and noisy data, provide a clear demonstration of the applicability and accuracy of the proposed method.
\end{abstract}
\maketitle


\noindent {{\bf AMS Subject Classifications.} 26A33, 35R11, 35R30, 74S40.

\section{Introduction} \label{(sec1)}
The quasi-static mathematical model of the elastoplastic torsion of a strain hardening bar is given by the following nonlinear boundary value problem \cite{kac}:
\begin{eqnarray}\label{S1-5}
\left \{ \begin{array}{l}
-\nabla \cdot \big(k(|\nabla u|^{2}) \nabla u \big)= 2\varphi, \, x \in\Omega\subset \mathbb{R}^2, \\
u(x)=0, \,  x \in \partial \Omega,
\end{array} \right.
\end{eqnarray}
where $\Omega \subset \mathbb{R}^2$ is the cross section of a bar, $\varphi$ is the angle of twist per unit length, $T:=|\nabla u|$ is the stress intensity, $u(x)$ is the Prandtl stress function and $k(T^2)$ is the plasticity function that satisfies the following conditions \cite{kac}, \cite{lan}:
\begin{eqnarray}\label{s1-3}
\left\{ \begin{array}{lc}
0 < c_0 \leq k(T^2) \leq c_1,\\
c_2 \leq k(T^2)+2k'(T^2)T^2 \leq c_3, \forall T^2 \in \big [{T_*}^2,{T^*}^2 \big],\\
k'(T^2)\leq 0,  \\
\exists {T_0}^2 \in ({T_*}^2,{T^*}^2):~k(T^2) = \frac {1} {G} ,~\forall T^2 \in \big [{T_*}^2, {T_0}^2 \big],
\end{array}\right.
\end{eqnarray}
where $T_0:=\max_{x \in \overline{ \Omega}} \vert \nabla u(x)\vert$ corresponds to the yield stress which is the maximum stress or force per unit area within a  material that can arise before the onset of permanent deformation, $\frac {1} {G}$ is the shear compliance, $G = \frac {E} {2(1+\nu)}$ is the elastic shear modulus, $E>0$ is the Young's modulus, $\nu$ is the Poisson coefficient. For any angle $\varphi>0$, all points of the bar have non-zero strain intensity which means the  ${T_*}^2>0$ in (\ref{s1-3}) makes sense. The last condition in (\ref{s1-3}) means that the elastic deformations precede the plastic ones.  The first two conditions in (\ref{s1-3}) imply the ellipticity of the nonlinear equation in (\ref{S1-5})  and hence positivity of the function $u(x), x \in  \Omega. $ The third condition  in (\ref{s1-3}) follows from the concavity of the curve $k(T^2) T$. A set that satisfies the conditions (\ref{s1-3}) and the condition $k \in L_\infty (0, {T^*}^2 )$ is denoted by   $\mathbb{K}$. We note that the following inequality is satisfied for a function $k  \in \mathbb{K} $ \cite{Arzu} :
\begin{equation}\label{s1-4}
 \left[  k(T^2)T - k(\tilde{T}^2)\tilde{T}\right]\cdot(T - \tilde{T} )\geq \tilde{C} |T - \tilde{T}|^2, \, \tilde{C} > 0.
\end{equation}

For many engineering materials, the function  $k(T^2)$ in (\ref{S1-5}) has the following form:
\begin{eqnarray}\label{engmaterial}
k(T^2) = \left \{ \begin{array}{ll}
1/G, \quad \quad  \quad \quad \quad \quad  \quad  \quad \quad   T^2 \leq {T_0}^2,  \\
1/G~\left (T^2 / {T_0}^2 \right )^{0.5(\kappa - 1)}, \, \quad  {T_0}^2<T^2,
\end{array} \right .
\end{eqnarray}
which corresponds to the Ramberg-Osgood curve $\sigma_i=\sigma_0(e_i/e_0)^\kappa$, where $\kappa \in (0,1)$ is the strain hardening exponent. The values $\kappa=1$ and $\kappa=0$ correspond to pure elastic and pure plastic cases, respectively. Evidently, this function satisfies all conditions in (\ref{s1-3}).\\

The direct problem (\ref{S1-5}) and the corresponding inverse coefficient problem have been well-studied both theoretically and numerically in the mathematical literature, \cite{Arzu} - \cite{has}.  However the real torsion process is not quasi-static, it depends on the time. The mathematical model of the real torsion process is given by the following evolutional equation:

\begin{eqnarray}\label{s1-6}
u_t-\nabla \cdot \big(k(T^2) \nabla u \big)= 2t,~(x, t) \in \Omega_\mathcal{T},
\end{eqnarray}
where $\Omega_\mathcal{T} := \Omega \times (0,\mathcal{T})$ and $\mathcal{T}>0$ is a final time. The direct and the corresponding inverse coefficient problem for the equation (\ref{s1-6}) have been well-studied both theoretically and numerically in the mathematical literature, \cite{tatar10}, \cite{tatar11}. The inverse coefficient problems for the equations in (\ref{S1-5}) and (\ref{s1-6}) consist of determining $\left (u(x, t) \, ,k(T^2) \right )$ from measured values of the torque  $\mathcal{M}_i:= \mathcal{M} (\varphi_i )=2\int \limits_{\Omega} u(x ; \varphi_i)dx$ and $\mathcal{M}_i:= \mathcal{M} (t_i )=2\int \limits_{\Omega} u(x ; t_i)dx$ respectively,  $i = 1, \dots , N $, where $N$ is the number of measurements. \\

We now motivate the use of fractional time derivative for our model by giving some examples in the literature. The equation (\ref{s1-6})  is known as Perona-Malik(PM) model and was first proposed in 1990, see \cite{perona}. This  equation uses inhomogeneous diffusivity coefficient and  widely used in image processing for purposes like smoothing, restoration, segmentation, filtering or detecting edges. Perona and Malik  proposed this model because they wanted to avoid localization problems of linear diffusion filtering. For this aim,  an inhomogeneous process is applied at locations which have a larger likelihood to be edges.  This likelihood is measured by $\vert \nabla u \vert^2$. The fractional order PDEs and their applications have been investigated in image processing community since fractional order equations could capture more features of images, especially the texture and other details due to the non-local property of fractional derivative. The fractional equations  are Euler-Lagrange equations of a cost functional which is an increasing function of the absolute value of the fractional derivative of the image intensity function, so these equations can be seen as generalizations of second-order and fourth-order anisotropic diffusion equations. In \cite{bai}, the authors introduced a new class of fractional-order anisotropic diffusion equations for noise removal based on PM equation for image denoising, where fractional order derivative was constructed by the fast Fourier transform algorithm. In \cite{wei}, the authors proposed a class of fractional order multi-scale variational models for image denoising and analyzed some properties of the fractional order total variation operator and its conjugate operator.   The authors proposed a new variation denoising model, where fractional order derivative was introduced into fidelity term to measure the similarity in the variation of images. The model can prevent the staircase effect and simultaneously enhance the noisy image because of the long-term memory and nonlocality of the fractional derivative in \cite{mako}.  We refer the readers to  \cite{osher1}-\cite{pu} and some of the references cited therein for more motivation on fractional models in image processing. Fractional models are also used in mechanics and material sciences. The recovery of the original dimensions of a deformed body when the load is removed is known as elastic behavior. If the elastic limit is exceeded, the body will experience a permanent set or deformation when the load is removed. A body which is permanently deformed is said to have undergone plastic deformation. For most materials, as long as the load does not exceed the elastic limit, the deformation is proportional to the load. Elastic behaviour  is modeled by Hooke's law $\sigma = E \varepsilon$ or $\frac {d^1\sigma}{d\varepsilon^1} = E$, where $\sigma$ is the normal stress and $\varepsilon$ is the normal strain. Plastic  behaviour is represented by $\sigma = \sigma_0$ or $\frac {d^0\sigma}{d\varepsilon^0} = \sigma_0$, where $\sigma_0$ is the elastic limit. Then, elastoplasticity may be modeled by $\frac {d^\beta\sigma}{d\varepsilon^\beta} = C,  0< \beta < 1$. Furthermore,  according to Newton's law, viscosity is represented as the first order derivative with respect time $t$, that is $\sigma  = c\frac {d\varepsilon} {dt}$. Considering this together with the Hooke's law, viscoelasticity, which is a combination of elastic and viscous behaviours, may be represented by $\sigma  = a\frac {d^\beta\varepsilon} {dt^\beta}, 0< \beta < 1$. Many papers can also be found related to viscoelasticity, viscoplasticity and viscoelastoplasticity as applications of fractional calculus, see \cite{men} - \cite{sun} and the references cited therein. Motivated by the aforementioned works, the authors in \cite{srm} studied the direct and the inverse problem for the following equation that may model viscoelastic torsion:
\begin{eqnarray}\label{s1-7}
{D_t}_ {0^+}^{\beta} u(x, t) -\nabla \cdot \big(k(T^2) \nabla u \big)= 2t,
\end{eqnarray}
where ${D_t}_ {0^+}^{\beta} u$ is the left-sided Caputo fractional derivative. The left-sided and the right-sided Caputo fractional derivatives are defined as follows, respectively:
\begin{eqnarray}\label{s1-1}
\left\{ \begin{array}{lc}
{D_t}_ {a^+}^{\beta} u:= \frac{1}{\bold \Gamma(1 - \beta)} \int_a^t (t-\tau)^{-\beta} \frac{\partial}{\partial \tau } u(x,\tau) d \tau,\\
{D_t}_ {b^-}^{\beta} u:= -  \frac{1}{\bold \Gamma(1 - \beta)} \int_t^b (\tau - t)^{-\beta} \frac{\partial}{\partial \tau } u(x,\tau) d \tau,
\end{array}\right.
\end{eqnarray}
where $\bold \Gamma$ is the Gamma function. We note that there is another kind of fractional derivative used frequently and is called the Riemann-Lioville fractional derivative. We note that the left-sided Caputo and Riemann-Lioville fractional derivatives agree if $u(x, a) = 0$ and the right-sided Caputo and Riemann-Lioville fractional derivatives agree if $u(x, b) = 0$. We refer the readers to \cite{Kil} and \cite{11} for further properties of the Caputo and Riemann-Lioville fractional derivatives. We note that the following fractional integration by parts formula holds provided that Caputo and Riemann-Lioville fractional derivatives agree, see \cite{Samko}, page 46 : 
\begin{equation}\label{s1-2}
 \int_a^b {D_t}_ {a^+}^{\beta} u(x, t) \,\, v(x, t) \,dt  =   \int_a^b u(x, t) \,\,  {D_t}_ {b_-}^{\beta} v(x, t) \, dt,
\end{equation}
where $u(x, t)$ and $v(x, t)$ are continuous on $\Omega_\mathcal{T}$ , ${D_t}_ {a^+}^{\beta} u(x, t)$ and ${D_t}_ {b_-}^{\beta} v(x, t)$ exist at every point and continuous  on $\Omega_\mathcal{T}$.  The direct and the inverse problem for the equation (\ref{s1-7}) are well-studied both numerically and theoretically, see \cite{srm}, \cite{stsuf}. The inverse problem in these papers consists of determining $\left (u(x, t) \, ,k(T^2) \right )$ from the measured values of the torque. In \cite{srm}, the authors solve the direct problem by the method of lines, then  a numerical method based on discretization of the minimization problem, steepest descent method, and least-squares approach is proposed for the solution of the inverse problem.  In \cite{stsuf}, the authors prove that the direct problem has a unique solution. Afterwards they prove the continuous dependence of the solution of the corresponding direct problem on the coefficient, the existence of a quasi-solution of the inverse problem is obtained in the appropriate class of admissible coefficients. An inverse problem is also studied for $\beta = 1$ in \cite{burhan} . In the current paper, we study the heat flux identification problem  for the equation (\ref{s1-7}) replacing $2t$ by a general source function $F(x, t)$ based on the boundary measurements. Recently, there has been a growing interest in inverse problems with fractional derivatives. Usually, in these works a fractional time derivative is considered and either determination of that or a source term or a coefficient under some additional condition(s) is the inverse problem. These problems are physically and practically very important. We list some of the   references  \cite{CNYY}-\cite{stsu}. Our paper can be regarded as an addition to these studies. \\   

This paper is organized as follows: In the section \ref{sec2}, we formulate the direct and the inverse problems, it is also proved that the direct problem is well-posed. In section \ref{sec3}, we linearize the nonlinear problem to a linear one and derive integral relationship between solutions of the introduced adjoint problems and the direct problem, then we prove the Fr\'echet  differentiability of the cost functional.  In section \ref{sec4},  a conjugate gradient method is given for the numerical solution of the inverse problem. The numerical solutions for both the direct and the inverse problems are addressed in section \ref{sec5}. 

\section{The direct and the inverse problems}\label{sec2}
In this section, we formulate the direct and the inverse problems. We also prove that the direct problem has a unique solution and the solution depends continuously on the fluxes. Then we prove that the inverse problem has a quasi-solution. Throughout the paper while, $\Vert \cdot \Vert$  denotes the usual $L^2(\Omega)-$norm, $\Vert \cdot \Vert_X$ denotes the norm in a Hilbert space $X$.  We first consider the following problem:
\begin{eqnarray}\label{s2-1}
\left \{ \begin{array}{l}
\displaystyle  {D_t}_ {0^+}^{\beta} u = \nabla \cdot \big(k(\vert \nabla u \vert ^2) \nabla u \big) + F(x, t),~(x,t)\in \Omega_\mathcal{T}, \\
- k(\vert \nabla u \vert ^2) \frac {\partial u}{\partial n} = f_1(x, t), ~(x,t)\in \Gamma_1 ^ \mathcal{T},  \\
- k(\vert \nabla u \vert ^2) \frac {\partial u}{\partial n} = f_2(x, t), ~(x,t)\in \Gamma_2 ^ \mathcal{T}, \\
u(x,t) = 0, ~(x,t)\in \Gamma_3^ \mathcal{T} \cup \Gamma_4^ \mathcal{T}, \\
u(x,0) = 0, \quad x \in \Omega,\\
\end{array} \right.
\end{eqnarray}

\noindent where $\beta \in (0, 1)$ is the fractional order of the time derivative,  $x := (x_1, x_2)$, $\mathcal{T} > 0$ is a final time, $\Omega := (0, 1) \times (0, 1)$,  $\Omega_ \mathcal{T} := \Omega \times (0, \mathcal{T})$, $\Gamma_i ^ \mathcal{T} := \Gamma_i \times (0, \mathcal{T}), i = 1, 2, 3, 4$,  $\Gamma_1 := \{0\} \times (0, 1)$, $\Gamma_2 := (0, 1) \times \{0\}$,  $\Gamma_3 := \{1\} \times (0, 1)$, $\Gamma_4 := (0, 1) \times \{1\}$ and $\Gamma := \bigcup \limits_{i=1}^{4} \Gamma_{i}$, $\Gamma_i  \cap \Gamma_j  = \emptyset, i \neq j$, $meas(\Gamma_i) \neq 0$,  is a piecewise smooth boundary of $\Omega$. For given $\beta \in (0, 1)$, $k \in \mathbb{K}$, $F(x, t) \in L_2(\Omega_ \mathcal{T})$, $f := (f_1(x, t), f_2(x, t)) \in L_2(\Gamma_1 ^ \mathcal{T}) \times  L_2(\Gamma_2 ^ \mathcal{T}) $, the problem (\ref{s2-1}) is called the direct problem. \\ 
 
Next we define the inverse problem. The inverse problem here consists of determining the unknown fluxes $f = (f_1(x, t), f_2(x, t))$ in the problem (\ref{s2-1}) from the following additional (measured output) data:
 \begin{eqnarray}\label{s211_1}
 \left \{ \begin{array}{l}
h_1(x, t) = u(x,t),  \, (x, t) \in \Gamma_1 ^ \mathcal{T},  \\
h_2(x, t) = u(x, t),  \, (x, t) \in \Gamma_2 ^ \mathcal{T}.
\end{array} \right.
\end{eqnarray}
That is, the problem (\ref{s2-1})-(\ref{s211_1}) is called the inverse problem for the given inputs $k, F(x, t), h := (h_1(x, t), h_2(x, t))$ and $\beta$. For the consistency of the  conditions in (\ref{s2-1})  and (\ref{s211_1}), it is assumed that $- k(\vert \nabla h_1 \vert ^2) \frac {\partial h_1}{\partial n} = f_1(x, t), ~(x,t)\in \Gamma_1 ^ \mathcal{T}$ and $- k(\vert \nabla h_2 \vert ^2) \frac {\partial h_2}{\partial n} = f_2(x, t), ~(x,t)\in \Gamma_2 ^ \mathcal{T}$.  We define the class of admissible fluxes $f := (f_1(x, t), f_2(x, t)) \in L_2(\Gamma_1 ^ \mathcal{T}) \times  L_2(\Gamma_2 ^ \mathcal{T}) $ for the inverse problem. Let  $W : = \mathcal{F}_1 \times \mathcal{F}_2 \subseteq L_2(\Gamma_1 ^ \mathcal{T}) \times L_2(\Gamma_2 ^ \mathcal{T}) $. W is called the class of admissible fluxes if the following conditions are satisfied: 
\begin{equation}\label{s21_1} 
\left \{ \begin{array}{l}
- \infty < \underline{F1} < f_1(x, t) < \overline{F1} < \infty, \\
- \infty < \underline{F2} < f_2(x, t) < \overline{F2} < \infty,
\end{array} \right.
\end{equation}
\noindent where $\underline{F1}, \, \overline{F1}, \, \underline{F2}, \,  \overline{F2}$ are constants. Evidently, $W \subseteq L_2(\Gamma_1 ^ \mathcal{T}) \times L_2(\Gamma_2 ^ \mathcal{T})$ is a closed and convex subset.\\

 For each $t \in (0, \mathcal{T})$, the inverse problem can be reformulated as a solution of the following nonlinear functional equation
\begin{equation}\label{s211_3}
u(x, t; f) = h, ~ x \in \Gamma_1 \cup \Gamma_2,
\end{equation}
where $u(x, t; f) $ denotes the solution of (\ref{s2-1}) for a given $f \in W$. However, due to measurement errors, in practice exact equality in (\ref{s211_3})  is usually not achieved. Hence, one needs to introduce the following auxiliary (cost) functional:
\begin{equation}\label{s2111_1}
J(f) = \int_0^\mathcal{T} \int \limits_{\Gamma_1}   \bigg \vert   u(x, t; f)  -  h_1(x, t) \bigg\vert^2 \, dx \, dt + \int_0^\mathcal{T} \int \limits_{\Gamma_2}   \bigg \vert   u(x, t; f)  -  h_2(x, t) \bigg\vert^2 \, dx \, dt,
\end{equation}
and consider the following minimization problem:
\begin{equation}\label{s21111_1}
J(f_\star)=\inf_{f \in {W}} J(f).
\end{equation}
A solution of the minimization problem (\ref{s21111_1}) is called a quasi-solution to the inverse problem. Evidently, the existence of a quasi-solution depends on the compactness of the class of admissible fluxes $W$ and the continuity of the functional (\ref{s2111_1}).
\begin{definition}
A weak solution of the problem (\ref{s2-1})  is a function $u \in L^2\left( 0, \mathcal{T}; H_0^1(\Omega) \right) \cap  W_2^\beta \left( 0, \mathcal{T}; L^2(\Omega) \right)$   such that the following integral identity holds:
\begin{equation}\label{s2-2} 
\int \limits_{\Omega_t}{D_t}_ {0^+}^{\beta}u \, \, v \, dx \,dt    +  \int \limits_{\Omega_t} k\left(|\nabla u|^2\right)  \nabla u \cdot \nabla v \, dx \,dt   = \int \limits_{\Omega_t} F(x, t) v  \,dx  \,dt
+ \int_0^t  \int _{\Gamma_1} f_1(x, t) v  \, dx  \, dt +\int_0^t  \int _{\Gamma_2} f_2(x, t) v  \, dx  \, dt,    
\end{equation}
for each $ v \in L^2\left( 0, \mathcal{T}; H_0^1(\Omega) \right) \cap  W_2^\beta \left( 0, \mathcal{T}; L_2(\Omega) \right)$, where  $W_2^\beta(0,\mathcal{T}):=\bigg \{u \in L^2[0,\mathcal{T}]:{D_t}_ {0^+}^{\beta} u  \in L^2[0,\mathcal{T}] \, \text{and}  \, u(0)=0 \bigg  \}$  is the fractional Sobolev space of order $\beta$. \end{definition}
\begin{theorem}\label{theorem1}

Let $k \in \mathbb{K}$. Then the direct problem (\ref{s2-1}) has a unique weak solution $u \in L^2\left( 0, T; H_0^1(\Omega) \right) \cap  W_2^\beta \left( 0, T; L^2(\Omega) \right)$. Moreover,  there exist some constants  $c, C > 0$ such that

\begin{equation}\label{s2-3} 
\int_0^t {D_t}_ {0^+}^{\beta} \Vert u \Vert ^2 \,dt + c \, \Vert u\Vert_{H_0^1(\Omega_t)}^2  \, \leq   C\,\left[   \Vert F\Vert_{L_2(\Omega_t)}^2+  \Vert f_1\Vert_{L_2(\Gamma_1 ^ t)}^2 +  \Vert f_2\Vert_{L_2(\Gamma_2 ^ t)}^2\right].
\end{equation}
\end{theorem}
\begin{proof}
We refer the readers to \cite{stsuf} for the existence and the uniqueness of the solution. We prove (\ref{s2-3}) here. Let  $u$ be a weak solution of (\ref{s2-1}) and let us take $v = u$  in the weak formulation (\ref{s2-2}). We then have:
\begin{equation}\label{s2-4} 
 \int \limits_{\Omega_t}{D_t}_ {0^+}^{\beta}u \, \, u \, dx \,dt   +  \int \limits_{\Omega_t} k\left(|\nabla u|^2\right)  |\nabla u|^2  \, dx \,dt   = \int \limits_{\Omega_t} F(x, t) u  \,dx  \,dt 
 + \int_0^t  \int _{\Gamma_1} f_1(x, t) u dx dt +\int_0^t  \int _{\Gamma_2} f_2(x, t) u dx dt.   
\end{equation}
By the first assumption in  (\ref{s1-3}), it follows that
\begin{equation}\label{s2-5}
c_0 \int \limits_{\Omega_t}   \vert \nabla u\vert^2 \, dx \, dt  \leq  \int \limits_{\Omega_t} k\left(|\nabla u|^2\right)  \vert \nabla u\vert^2 \, dx \, dt .
\end{equation}
Moreover, using Cauchy-Schwarz (CS), Young (Y), Poincar\'e (P) inequalities and trace theorem (T), we have:
\begin{equation}\label{s2-6}
\begin{split}
&\int_0^t  \int_{\Omega} F(x, t) u  \,dx  \,dt + \int_0^t  \int _{\Gamma_1} f_1(x, t) u\, dx\, dt +\int_0^t  \int _{\Gamma_2} f_2(x, t) u\, dx\, dt  \\
& \overset{\text{(CS)}}{\leq}  c_4\int _0^t \Vert F \Vert  \Vert u \Vert \,dt +  c_5\int _0^t \Vert f_1 \Vert_{L_2 {(\Gamma_1})}  \Vert u \Vert_{L_2 {(\Gamma_1})} \,dt + c_6\int _0^t \Vert f_2 \Vert_{L_2 {(\Gamma_2})} \Vert u \Vert _{L_2 {(\Gamma_2})} \,dt \\
&  \leq c_4\int _0^t \Vert F \Vert  \Vert u \Vert \,dt + c_5 \int _0^t \Vert f_1 \Vert_{L_2 {(\Gamma_1})}  \Vert u \Vert_{L_2 {(\partial \Omega})} \,dt+ c_6\int _0^t \Vert f_2 \Vert_{L_2 {(\Gamma_2})} \Vert u \Vert _{L_2 {(\partial \Omega})} \,dt \\
&   \overset{\text{(T)}} \leq c_4\int _0^t \Vert F \Vert  \Vert u \Vert \,dt + c_7 \int _0^t \Vert f_1 \Vert_{L_2 {(\Gamma_1})}  \Vert u \Vert_{H^1 {( \Omega})} \,dt+ c_8\int _0^t \Vert f_2 \Vert_{L_2 {(\Gamma_2})} \Vert u \Vert _{H^1 {(\Omega})} \,dt \\
&   \overset{\text{(Y)}} \leq c_9\int _0^t ( \left\Vert F \Vert^2 + \Vert u \Vert^2 \right) \,dt + c_{10} \int _0^t  \left ( \Vert f_1 \Vert^2_{L_2 {(\Gamma_1})} + \Vert u \Vert^2_{H^1 {( \Omega})}\right)  \,dt + c_{11}\int _0^t  \left ( \Vert f_2 \Vert^2_{L_2 {(\Gamma_2})} + \Vert u \Vert^2_{H^1 {( \Omega})}\right)  \,dt  \\
&   \overset{\text{(P)}} \leq c_{12}\int _0^t (\left \Vert F \Vert^2 +  \Vert \nabla u \Vert^2 \right ) \,dt + c_{13} \int _0^t \left ( \Vert f_1 \Vert^2_{L_2 {(\Gamma_1})} + \Vert \nabla u\Vert ^2 \right) \,dt + c_{14}\int _0^t \left( \Vert f_2 \Vert_{L_2 {(\Gamma_2})} + \Vert \nabla u\Vert ^2 \right) \,dt \\
&  \leq c_{15} \Vert \nabla u\Vert ^2 + c_{16} \left ( \Vert F \Vert _{L_2(\Omega_t)}^2 + \Vert f_1 \Vert ^2_{L_2 {( \Gamma_1^t})} + \Vert f_2 \Vert ^2_{L_2 {( \Gamma_2^t})} \right ) .
\end{split}
\end{equation}
Now, employing the Alikhanov inequality \cite{Ali} we have:
\begin{equation}\label{s2-7}
\int \limits_{\Omega_t}  {D_t}_ {0^+}^{\beta}u  \, u \, dx \geq \frac{1}{2} \int_0^t {D_t}_ {0^+}^{\beta} \Vert u \Vert ^2 \,dt .
\end{equation}
Using (\ref{s2-5}), (\ref{s2-6}) and  (\ref{s2-7}) in (\ref{s2-4}) we get
\begin{equation*}
\int_0^t {D_t}_ {0^+}^{\beta} \Vert u \Vert ^2 \,dt  + c  \Vert \nabla u\Vert ^2 \leq C \left ( \Vert F \Vert_{L_2(\Omega_t)}^2  + \Vert f_1 \Vert ^2_{L_2 {( \Gamma_1^t})} + \Vert f_1 \Vert ^2_{L_2 {( \Gamma_2^t})} \right ) .
\end{equation*}
This completes the proof.
\end{proof}

\begin{theorem}\label{conv}
Suppose that a sequence of flux function $f_m := (f_{1m}, f_{2m}) \in L_2(\Gamma_1 ^ \mathcal{T}) \times L_2(\Gamma_2 ^ \mathcal{T}) $ converges pointwise to a function $f = (f_1, f_2) \in L_2(\Gamma_1 ^ \mathcal{T}) \times L_2(\Gamma_2 ^ \mathcal{T}) $. Then the sequence of solutions $u_m := u(x, t; f_m)$ converges to the solution $u := u(x, t; f)  \in L^2\left( 0, T; H_0^1(\Omega) \right) \cap  W_2^\beta \left( 0, T; L^2(\Omega) \right) $, where $u := u(x, t; f) $ denotes the solution of the direct problem (\ref{s2-1}) for a given  flux $f$.
\end{theorem}
\begin{proof}
By Theorem (\ref{theorem1}), the solutions $u_m$ and $u$ are well-defined. By (\ref{s2-2}), we have:
\begin{equation}\label{s2-8} 
\int \limits_{\Omega_t}{D_t}_ {0^+}^{\beta}u_m \, \, v \, dx \,dt   +  \int \limits_{\Omega_t} k\left(|\nabla u_m|^2\right)  \nabla u_m \cdot \nabla v \, dx \,dt   = \int \limits_{\Omega_t} F(x, t) v  \,dx  \,dt 
 + \int_0^t  \int _{\Gamma_1} f_1(x, t) v \, dx \, dt +\int_0^t  \int _{\Gamma_2} f_2(x, t) v \, dx \, dt .
\end{equation}
\begin{equation}\label{s2-9} 
\int \limits_{\Omega_t}{D_t}_ {0^+}^{\beta}u \, \, v \, dx \,dt   +  \int \limits_{\Omega_t} k\left(|\nabla u|^2\right)  \nabla u \cdot \nabla v \, dx \,dt   = \int \limits_{\Omega_t} F(x, t) v  \,dx  \,dt 
+ \int_0^t  \int _{\Gamma_1} f_1(x, t) v \, dx \, dt +\int_0^t  \int _{\Gamma_2} f_2(x, t) v \, dx \, dt.
\end{equation}
By taking $v := u_m - u$, we have from (\ref{s2-8}) and (\ref{s2-9}) that
\begin{equation}\label{s2-10} 
\begin{split}
&\int \limits_{\Omega_t}{D_t}_ {0^+}^{\beta} (u_m - u) \, (u_m - u) dx dt  +  \int \limits_{\Omega_t} \left(k\left(|\nabla u_m|^2\right) \nabla u_m - k\left(|\nabla u|^2\right)  \nabla u \right) \cdot \nabla(u_m - u)dx dt \\
 &=  \int_0^t  \int _{\Gamma_1} (f_{1m} - f_1) (u_m - u) \, dx \, dt + \int_0^t  \int _{\Gamma_2} (f_{2m} - f_2) (u_m - u) \, dx \, dt ,\\
 & \overset{\text{(CS)}}{\leq} \Vert f_{1m} - f_1\Vert_{L_2 {(\Gamma_1^t})} \Vert u_m - u \Vert_{L_2 {(\Gamma_1^t})}  +  \Vert f_{2m} - f_2\Vert_{L_2 {(\Gamma_2^t})} \Vert u_m - u \Vert_{L_2 {(\Gamma_2^t})}  .
 \end{split}
\end{equation}
By using  (\ref{s1-4}) and (\ref{s2-7}) in (\ref{s2-10}) , we deduce that
\begin{equation}\label{s2-11} 
\frac{1}{2} \int_0^t {D_t}_ {0^+}^{\beta} \Vert u_m - u \Vert^2 \,dt   +  \tilde{C} \Vert \nabla (u_m - u) \Vert^2 \leq \Vert f_{1m} - f_1\Vert_{L_2 {(\Gamma_1^t})} \Vert u_m - u \Vert_{L_2 {(\Gamma_1^t})} +  \Vert f_{2m} - f_2\Vert_{L_2 {(\Gamma_2^t})} \Vert u_m - u \Vert_{L_2 {(\Gamma_2^t})}.
\end{equation}
If we pass the limit when $m \to \infty $ in (\ref{s2-11}), it is easy to see that
\begin{equation*}
 \lim_{m \to \infty} {D_t}_ {0^+}^{\beta} \Vert u_m - u \Vert ^2 = 0, \; \; \text{ and } \; \lim_{m \to \infty}  \big \Vert \nabla \left( u_m - u \right) \big \Vert  = 0.
\end{equation*}
\end{proof}
\begin{remark}
One can easily prove that, by the Theorem (\ref{conv}),  if $f_m := (f_{1m}, f_{2m}) \in L_2(\Gamma_1 ^ \mathcal{T}) \times L_2(\Gamma_2 ^ \mathcal{T}) $ converges pointwise to a function $f = (f_1, f_2) \in L_2(\Gamma_1 ^ \mathcal{T}) \times L_2(\Gamma_2 ^ \mathcal{T}) $, then $J(f_m)$ converges to  $J(f)$. By the Weierstrass existence theorem, the set of solutions
\begin{equation*}
W_\star:= \{f \in W : J(f_\star) = J_\star =\inf J(f) \}
\end{equation*}
is not an empty set. 
\end{remark}
\section{Adjoint problem approach for linearized inverse problem and Fr\'echet differentiability of the cost functional}\label{sec3}
In this section, we linearize the nonlinear problem (\ref{s2-1}) using adjoint problem technique.  Let $u^{(n)} (x, t)$ be the solution of the following linearized problem: 
\begin{eqnarray}\label{s3-1}
\left \{ \begin{array}{l}
\displaystyle  {D_t}_ {0^+}^{\beta} u^{(n)} = \nabla \cdot \big(k(\vert \nabla u^{(n-1)} \vert ^2) \nabla u^{(n)} \big) + F(x, t),~(x,t)\in \Omega_\mathcal{T}, \\
- k(\vert \nabla u^{(n-1)} \vert ^2) \frac {\partial u^{(n)}}{\partial n} = f_1(x, t), ~(x,t)\in \Gamma_1 ^ \mathcal{T},  \\
- k(\vert \nabla u^{(n-1)} \vert ^2) \frac {\partial u^{(n)}}{\partial n} = f_2(x, t), ~(x,t)\in \Gamma_2 ^ \mathcal{T}, \\
u^{(n)}(x,t) = 0, ~(x,t)\in \Gamma_3^ \mathcal{T} \cup \Gamma_4^ \mathcal{T}, \\
u^{(n)}(x,0) = 0, \quad x \in \Omega,\\
\end{array} \right.
\end{eqnarray}
where $u^{(0)}(x,t) \equiv 0 $ is a chosen initial function. By following the same idea in the Theorem (\ref{conv}), we conclude that the sequence $u^{(n)} (x, t)$ is bounded in $L^2\left( 0, \mathcal{T}; H_0^1(\Omega) \right) $. Then there is a subsequence (still denoted by $u^{(n)} (x, t)$)   such that $u^{(n)} \to \tilde{u}, \, n \to \infty$ for some function $\tilde{u}$.
This implies that 
\begin{equation}\label{convergen}
\Vert u^{(n)} -  \tilde{u}\Vert_{L^2\left( 0, \mathcal{T}; H_0^1(\Omega) \right)} \to 0, n \to \infty.
\end{equation}
Then the linearized problem (\ref{s3-1})  is transformed into the following linear problem :
\begin{eqnarray}\label{s3-2}
\left \{ \begin{array}{l}
\displaystyle  {D_t}_ {0^+}^{\beta} \tilde{u} = \nabla \cdot \big(\tilde{k} (x, t) \nabla \tilde{u} \big) + F(x, t),~(x,t)\in \Omega_\mathcal{T}, \\
- \tilde{k} (x, t) \frac {\partial \tilde{u}}{\partial n} = f_1(x, t), ~(x,t)\in \Gamma_1 ^ \mathcal{T},  \\
- \tilde{k} (x, t) \frac {\partial \tilde{u}}{\partial n} = f_2(x, t), ~(x,t)\in \Gamma_2 ^ \mathcal{T}, \\
\tilde{u}(x,t) = 0, ~(x,t)\in \Gamma_3^ \mathcal{T} \cup \Gamma_4^ \mathcal{T}, \\
\tilde{u}(x,0) = 0, \quad x \in \Omega,\\
\end{array} \right.
\end{eqnarray}
where $ \tilde{k} (x,t) := k\big(\vert \nabla u^{(n-1)}\big)$. Then we define the additional data (\ref{s211_1}) and the cost functional (\ref{s2111_1}) for the linear problem (\ref{s3-2}) as follows:
 \begin{eqnarray}\label{s3-3}
 \left \{ \begin{array}{l}
h_1(x, t) = \tilde{u}(x,t),  \, (x, t) \in \Gamma_1 ^ \mathcal{T},  \\
h_2(x, t) = \tilde{u}(x, t),  \, (x, t) \in \Gamma_2 ^ \mathcal{T}. 
\end{array} \right.
\end{eqnarray}
\begin{equation}\label{s3-4}
\tilde{J}(f) = \int_0^\mathcal{T} \int \limits_{\Gamma_1}   \bigg \vert  \tilde{u} (x, t; f)  -  h_1(x, t) \bigg\vert^2 \, dx \, dt + \int_0^\mathcal{T} \int \limits_{\Gamma_2}   \bigg \vert   \tilde{u} (x, t; f)  -  h_2(x, t) \bigg\vert^2 \, dx \, dt.
\end{equation}
We calculate the first variation of the cost functional as follows. Let $\delta \tilde{u} := \tilde{u} (x, t; f + \delta f) - \tilde{u} (x, t; f ) $. Then we have: 
\begin{equation}\label{s3-5}
\begin{split}
\delta\tilde{J}(f) &= \tilde{J}(f + \delta f) - \tilde{J}(f)  = \int_0^\mathcal{T} \int \limits_{\Gamma_1}   \left ( \tilde{u} (x, t; f + \delta f)- h_1(x, t) \right ) ^2 dxdt - \int_0^\mathcal{T} \int \limits_{\Gamma_1} \left (   \tilde{u} (x, t; f)- h_1(x, t) \right )^2 dxdt.\\
& + \int_0^\mathcal{T} \int \limits_{\Gamma_2}   \left ( \tilde{u} (x, t; f + \delta f)- h_2(x, t) \right ) ^2 dx dt - \int_0^\mathcal{T} \int \limits_{\Gamma_2} \left (   \tilde{u} (x, t; f)- h_2(x, t) \right )^2 dxdt.\\
&= \int_0^\mathcal{T} \int \limits_{\Gamma_1}   \bigg [ (\tilde{u} (x, t; f + \delta f) - \tilde{u} (x, t; f ))  \underbrace { (\tilde{u}(x, t; f + \delta f) + \tilde{u} (x, t; f ))}_
{\delta  \tilde{u} + 2 \tilde{u}(x, t: f)  }+2h_1(\tilde{u} (x, t; f ) - \tilde{u} (x, t; f + \delta f)) \bigg ] dxdt\\
 &+ \int_0^\mathcal{T} \int \limits_{\Gamma_2}   \bigg [ (\tilde{u} (x, t; f + \delta f) - \tilde{u} (x, t; f ))   \underbrace { (\tilde{u}(x, t; f + \delta f) + \tilde{u} (x, t; f ))}_
{\delta  \tilde{u} + 2 \tilde{u}(x, t: f)  }+2h_2(\tilde{u} (x, t; f ) - \tilde{u} (x, t; f + \delta f)) \bigg ] dx dt\\
&= 2\int_0^\mathcal{T} \int \limits_{\Gamma_1} \delta{ \tilde{u}} \left (  (\tilde{u}(x, t ; f) - h_1(x, t) \right)dx dt + \int_0^\mathcal{T} \int \limits_{\Gamma_1} (\delta  \tilde{u})^2 dx dt\\
&+2\int_0^\mathcal{T} \int \limits_{\Gamma_2} \delta{ \tilde{u}} \left (  (\tilde{u}(x, t ; f) - h_2(x, t) \right)dxdt + \int_0^\mathcal{T} \int \limits_{\Gamma_2} (\delta  \tilde{u})^2 dx dt,\\
\end{split}
\end{equation}
where $\delta \tilde{u}$ is the solution of the following problem: 
\begin{eqnarray}\label{s3-6}
\displaystyle
\left \{ \begin{array}{l}
  {D_t}_ {0^+}^{\beta} \delta \tilde{u} = \nabla \cdot \big(\tilde{k} (x, t) \nabla \delta \tilde{u}\big), ~(x,t)\in \Omega_\mathcal{T}, \\
- \tilde{k} (x, t) \frac {\partial \delta \tilde{u}}{\partial n} = \delta f_1(x, t), ~(x,t)\in \Gamma_1 ^ \mathcal{T},  \\
- \tilde{k} (x, t) \frac {\partial \delta \tilde{u}}{\partial n} = \delta f_2(x, t), ~(x,t)\in \Gamma_2 ^ \mathcal{T}, \\
\tilde{u}(x,t) = 0, ~(x,t)\in \Gamma_3^ \mathcal{T} \cup \Gamma_4^ \mathcal{T}, \\
\tilde{u}(x,0) = 0, \quad x \in \Omega.\\
\end{array} \right.
\end{eqnarray}
\begin{theorem}
Let $\tilde{\varphi}$ be the solution of the following well-posed backward parabolic problem:
\begin{eqnarray}\label{s3-7}
\left \{ \begin{array}{l}
\displaystyle  {D_t}_ {\mathcal{T}^-}^{\beta} {\tilde{\varphi}} = \nabla \cdot \big(\tilde{k} (x, t) \nabla \tilde{\varphi} \big),~(x,t)\in \Omega_\mathcal{T}, \\
- \tilde{k} (x, t) \frac {\partial \tilde{\varphi}}{\partial n} = 2 \left (\tilde{u}(x, t; f) - f_1(x, t) \right ), ~(x,t) \in \Gamma_1 ^ \mathcal{T},  \\
- \tilde{k} (x, t) \frac {\partial \tilde{\varphi}}{\partial n} = 2 \left (\tilde{u}(x, t; f) - f_2(x, t) \right ), ~(x,t)\in \Gamma_2 ^ \mathcal{T}, \\
\tilde{\varphi} (x,t) = 0, ~(x,t)\in \Gamma_3^ \mathcal{T} \cup \Gamma_4^ \mathcal{T}, \\
\tilde{\varphi} (x,\mathcal{T}) = 0, \quad x \in \Omega,\\
\end{array} \right.
\end{eqnarray}
\end{theorem}
\noindent where $\tilde{u}(x, t; f)$ and $\tilde{u}(x, t; f + \delta f)$ are the solutions of the problem (\ref{s3-2}). Then the following integral identity holds:
\begin{equation}\label{s3-8}
\begin{split}
& 2 \int _0^\mathcal{T} \int_{\Gamma_1} \left (\tilde{u}(x, t; f) - h_1(x, t)  \right )  \delta \tilde{u} dx dt + 2 \int _0^\mathcal{T} \int_{\Gamma_2} \left (\tilde{u}(x, t; f) - h_2(x, t)  \right )  \delta \tilde{u} dx dt  \\
& = \int _0^\mathcal{T} \int_{\Gamma_1} \tilde{\varphi}(x,t)  \delta f_1 dx dt +  \int _0^\mathcal{T} \int_{\Gamma_2} \tilde{\varphi}(x,t)  \delta f_2 dx dt . \\
\end{split}
\end{equation}
\begin{proof}
We refer the readers to  \cite{Mis6} for the well-posedness of the problem (\ref{s3-7}).  We start the proof by multiplying the equation in (\ref{s3-6}) by an arbitrary function $\tilde{\varphi} (x, t)$ and integrating the resulting equation over $\Omega_\mathcal{T}$. Then we have:
\begin{equation}\label{s3-9}
\int_{\Omega_\mathcal{T}} {D_t}_ {0^+}^{\beta} \delta \tilde{u} \, \, \tilde{\varphi} (x, t) \,   dx\, dt  = \int_{\Omega_\mathcal{T}} \nabla. \big(\tilde{k} (x, t) \nabla \delta \tilde{u}\big) \,\tilde{\varphi} (x, t) \, dx\, dt .
\end{equation}
First, we apply the integration by parts formula to the left side in (\ref{s3-9}). By (\ref{s1-2}), $\tilde{\varphi} (x,\mathcal{T}) = 0$ and $\delta \tilde{u}(x, 0) = 0 $ we deduce that:
\begin{equation}\label{s3-10}
\begin{split}
&\int_{\Omega_\mathcal{T}} {D_t}_ {0^+}^{\beta} \delta \tilde{u} \, \, \tilde{\varphi} (x, t) \,   dx\, dt  = \int_{\Omega_\mathcal{T}}\delta \tilde{u}  \underbrace {{D_t}_ {\mathcal{T}-}^{\beta} \tilde{\varphi} }_{\nabla. \big(\tilde{k} (x, t) \nabla \tilde{\varphi} \big)} \,   dx\, dt  \\
& = \int_0^\mathcal{T} \int_0^1 \tilde{k} \tilde{\varphi}_{x_1}  \delta \tilde{u} \vert _0^1 \, dx_2 \, dt + \int_0^\mathcal{T} \int_0^1 \tilde{k} \tilde{\varphi}_{x_2}  \delta \tilde{u} \vert _0^1 \,  dx_1\,dt- \int_{\Omega_\mathcal{T}} \tilde{k} \nabla  \tilde{\varphi} \nabla \delta \tilde{u} \, dx \, dt  .\\
\end{split}
\end{equation}
Next, we apply the integration by parts formula to the right side in (\ref{s3-9}). Then, we have:
\begin{equation}\label{s3-11}
\int_{\Omega_\mathcal{T}} \nabla \cdot \big(\tilde{k} (x, t) \nabla \delta \tilde{u}\big) \,\tilde{\varphi} (x, t) \, dx\, dt =  \int_0^\mathcal{T} \int_0^1 \tilde{k} \delta \tilde{u}_{x_1} \tilde{\varphi}  \vert _0^1 \, dx_2 \, dt + \int_0^\mathcal{T} \int_0^1 \tilde{k} \delta \tilde{u}_{x_2} \tilde{\varphi}  \vert _0^1 \, dx_1 \, dt  - \int_{\Omega_\mathcal{T}} \tilde{k} \nabla  \tilde{\varphi} \nabla \delta \tilde{u} \, dx \, dt  .\\
\end{equation}
By (\ref{s3-9}), (\ref{s3-10}) and (\ref{s3-11}), we have:
\begin{equation}\label{s3-12}
\int_0^\mathcal{T} \int_0^1 \tilde{k} \tilde{\varphi}_{x_1}  \delta \tilde{u} \vert _0^1 \, dx_2 \, dt + \int_0^\mathcal{T} \int_0^1 \tilde{k} \tilde{\varphi}_{x_2}  \delta \tilde{u} \vert _0^1 \,  dx_1\,dt  =\int_0^\mathcal{T} \int_0^1 \tilde{k} \delta \tilde{u}_{x_1} \tilde{\varphi}  \vert _0^1 \, dx_2 \, dt + \int_0^\mathcal{T} \int_0^1 \tilde{k} \delta \tilde{u}_{x_2} \tilde{\varphi}  \vert _0^1 \, dx_1 \, dt .
\end{equation}
If we take the initial and boundary conditions in (\ref{s3-2}) and (\ref{s3-7}) into account, we complete the proof.
\end{proof}

\noindent Next we prove that $\tilde{J}(f)$ is Fr\'{e}chet differentiable. By using (\ref{s3-8}) in (\ref{s3-5}), we deduce that  
\begin{equation}\label{s3-13}
\delta\tilde{J}(f) = \int _0^\mathcal{T} \int_{\Gamma_1} \tilde{\varphi}(x,t)  \delta f_1 dx dt +  \int _0^\mathcal{T} \int_{\Gamma_2} \tilde{\varphi}(x,t)  \delta f_2 dx dt + \int_0^\mathcal{T} \int \limits_{\Gamma_1} (\delta  \tilde{u})^2 dx dt + \int_0^\mathcal{T} \int \limits_{\Gamma_2} (\delta  \tilde{u})^2 dx dt.
\end{equation}
Then by definition of Fr\'{e}chet differentiability
\begin{equation*}
 \delta\tilde{J}(f) = \tilde{J}(f + \delta f) - \tilde{J}(f) = \big ( \nabla \tilde{J}(f), \delta f \big) + o\big( \Vert \delta f \Vert^2 \big),
 \end{equation*}
 and using the following inequality proved in \cite{burhan}: 
\begin{equation}\label{s3-14}
 \int_0^\mathcal{T} \int \limits_{\Gamma_1} (\delta  \tilde{u})^2 dx dt + \int_0^\mathcal{T} \int \limits_{\Gamma_2} (\delta  \tilde{u})^2 dx dt \leq L  \int_0^\mathcal{T} \int \limits_{\Gamma_1} (\delta  f_1)^2 dx dt + \int_0^\mathcal{T} \int \limits_{\Gamma_2} (\delta  f_2)^2 dx dt,
\end{equation}
 we have the following theorem: 
\begin{theorem}
The cost functional $\tilde{J}(f)$ is Fr\'{e}chet differentiable and the  Fr\'{e}chet derivative at $f$ can be calculated via the solution of the problem (\ref{s3-7}) as follows:
\begin{eqnarray}
\nabla \tilde{J}(f) = \{\tilde{\varphi}(x, t ; f)_{|_{\Gamma_1 ^ \mathcal{T}}}, \tilde{\varphi}(x, t ; f)_{|_{\Gamma_2 ^ \mathcal{T}}} \}.
\end{eqnarray}
\end{theorem}
Since $W \subseteq L_2(\Gamma_1 ^ \mathcal{T}) \times L_2(\Gamma_2 ^ \mathcal{T})$ is a closed and convex subset and $\tilde{J}(f)$ is  convex functional, then the necessary and sufficient condition for $f_ \star:= (f_{1\star},  f_{2\star})$ to be a quasi-solution of the inverse problem (\ref{s3-2}) and  (\ref{s3-3}) is that the following variational inequality holds \cite{convex} :
 \begin{equation*}
\bigg (\nabla \tilde{J}(f_ \star), f - f_ \star \bigg ) \geq 0, ~ \forall ~  f \in W. 
\end{equation*}
\noindent Then we have the following theorem: 
\begin{theorem}
Let $\varphi$ be the solution to the problem (\ref{s3-7}) for a given $f \in W$. Then  $f_ \star$ is a quasi-solution to the inverse problem  (\ref{s3-2}) and  (\ref{s3-3}) if and only if the following variational inequality holds: 
 \begin{equation*}
\int _0^\mathcal{T} \int_{\Gamma_1} \tilde{\varphi}(x, t ; f)(f_1 - f_{1\star}) dx dt + \int _0^\mathcal{T} \int_{\Gamma_2} \tilde{\varphi}(x, t ; f)(f_2 - f_{2\star}) dx dt \geq 0.
\end{equation*}
\end{theorem}
\begin{theorem}\label{cont}
The functional $ \tilde{J}(f)$  is of Hölder class $C^{1,1}(W)$ and the following inequality holds:
\begin{eqnarray}\label{holder}
\Vert \nabla \tilde{J}(f+\delta f) - \nabla \tilde{J}(f) \Vert  \leq L \Vert f \Vert , L \ge 0, \forall ~ f, f+ \delta f \in W . 
\end{eqnarray}
\end{theorem}
\begin{proof}
Let $\tilde{\varphi}$ be the solution of the well-posed backward parabolic problem (\ref{s3-7}) and $\tilde{u}(x, t; f)$ and $\tilde{u}(x, t; f + \delta f)$ be the solutions of the problem (\ref{s3-2}). Then $\tilde{\varphi}$ is the solution of the following problem: 
\begin{eqnarray}\label{s3-15}
\left \{ \begin{array}{l}
\displaystyle  {D_t}_ {\mathcal{T}^-}^{\beta} \delta {\tilde{\varphi}} = \nabla \cdot \big(\tilde{k} (x, t) \nabla \tilde{\delta \varphi} \big),~(x,t)\in \Omega_\mathcal{T}, \\
- \tilde{k} (x, t) \frac {\partial \delta\tilde{ \varphi}}{\partial n} = \delta \tilde{u}(x, t; f), ~(x,t) \in \Gamma_1 ^ \mathcal{T},  \\
- \tilde{k} (x, t) \frac {\partial \delta \tilde{\varphi}}{\partial n} = \delta \tilde{u}(x, t; f), ~(x,t)\in \Gamma_2 ^ \mathcal{T}, \\
\delta \tilde{\varphi} (x,t) = 0, ~(x,t)\in \Gamma_3^ \mathcal{T} \cup \Gamma_4^ \mathcal{T}, \\
\delta \tilde{\varphi} (x,\mathcal{T}) = 0, \quad x \in \Omega .\\
\end{array} \right.
\end{eqnarray}
By multiplying the equation in (\ref{s3-15}) by $\delta \tilde{\varphi}$, integrating on $\Omega_\mathcal{T}$ and using the initial and boundary conditions, we obtain the following energy identity:
\begin{equation}\label{s3-16}
\int_{\Omega_\mathcal{T}} {D_t}_ {T^-}^{\beta} \delta \tilde{\varphi} \, \cdot \delta \tilde{\varphi}  (x, t) dx dt +\int_{\Omega_\mathcal{T}} \tilde k\left(|\nabla \delta \tilde{\varphi}|^2\right ) dx dt  =  \int_0^\mathcal{T} \int \limits_{\Gamma_1} (\delta  \tilde{u})^2 dx dt + \int_0^\mathcal{T} \int \limits_{\Gamma_2} (\delta  \tilde{u})^2 dx dt.
\end{equation}
By using Poincar\'e  inequality, the first condition in (\ref{s1-3}) and (\ref{s3-14}) in (\ref{s3-16}), we complete the proof. 
\end{proof}
\begin{remark}
Since $ \tilde{J}(f)$ is  Lipschitz continuous, the constant $\alpha_n > 0$ in the following iteration process based on any gradient method 
\begin{eqnarray}\label{iteration}
f_{n+1} = f_n - \alpha_n \nabla  \tilde{J}(f_n),
\end{eqnarray}
can be determined by $0 < \delta_0 \le \alpha_n \le \frac{2}{L + 2\delta_1} $, where $\delta_0$ and $\delta_1$ are arbitrary parameters. 
\end{remark}
The proofs of the following theorems (\ref{ah1}) - (\ref{ah3}) follow the ideas of \cite{hasanov1} closely. For the convenience of the readers, we sketch their proofs here. 
\begin{theorem}\label{ah1}
The following inequality holds $\forall ~  f_1, f_2 \in W$:
 \begin{equation}\label{s3-17}
\vert  \tilde{J}(f_1) - \tilde{J}(f_2) - (\nabla \tilde{J} (f_2), f_1 - f_2) \vert \leq  \frac{1}{2} L \Vert   f_1 - f_2\Vert^2. 
\end{equation}
\end{theorem}
\begin{proof}
By using (\ref{holder}) in the following formula
 \begin{equation*}
\big \vert \tilde{J}(f_1) - \tilde{J}(f_2) \big \vert = \bigg \vert \int_0^1 \nabla \tilde{J}(f_2 +\theta (f_1 - f_2) ) d\theta \bigg \vert,
\end{equation*}
we have the proof. 
\end{proof}
\begin{theorem}\label{ah2}
Let  $f_n \in W$ be iterations defined by (\ref{iteration}). Then the following inequality holds for $n = 0, 1, 2, \cdots$  :
 \begin{equation}\label{s3-20}
\tilde{J}(f_n) - \tilde{J}(f_{n+1}) \geq \frac{1}{2L} \Vert \nabla \tilde{J}(f_n) \Vert^2.
\end{equation}
\end{theorem}
\begin{proof}
Let $f_1 = f_{n+1}$, $f_2 = f_{n}$ and $\alpha_n = \alpha > 0$ in (\ref{s3-17}). Then we have:
 \begin{equation*}
\tilde{J}(f_n) - \tilde{J}(f_{n+1}) \geq \alpha (1 - \frac{\alpha L}{2}), \forall \alpha >0 . 
\end{equation*}
The maximum value of the function $\alpha (1 - \frac{\alpha L}{2})$  occurs at $\alpha_\star = \frac{1}{L} $, therefore the maximum value of that function is  $\frac{1}{2L}$. This completes the proof. 
\end{proof}
\begin{remark}
The optimal value of the iteration parameter $\alpha_n$ that corresponds to $\delta_0 = \frac{1} {L}$ and $\delta_1 = \frac{L} {2}$ is $\alpha_\star = \frac{1}{L} $.
\end{remark}
\begin{remark}\label{decroi}
Let $\{f_n\}$ be the sequence defined by (\ref{iteration}). Then, by (\ref{s3-20}),   $\{\tilde J(f_n)\}$ is a monotone decreasing convergent sequence and
 \begin{equation*}
\lim_{n \to \infty} \Vert \nabla \tilde J(f_n)  \Vert = 0.
\end{equation*}
Also, the following estimate holds: 
 \begin{equation*}
\Vert f_{n+1} - f_n \Vert ^2 \leq 2 \big (\tilde J(f_n) - \tilde J(f_{n+1})\big ). 
\end{equation*}
We note that for a detailed proof about the convergence and monotonicity of the sequence $\{\tilde J(f_n)\}$, one can refer to the work in \cite{lesnic}.
\end{remark}
\begin{theorem}\label{ah3}
Let  $f_0 \in W$ be an initial iteration. Then $\{f_n\} \subset W$, defined by (\ref{iteration}), converges to a quasi-solution $f_\star \in W$ of the problem (\ref{s3-2}) and the following estimate holds for $\tilde{J}(f_n)$:
 \begin{equation}\label{s3-22}
0 \leq \tilde{J}(f_n) - \tilde{J}(f_\star) \leq 2L d^2 n^{-1}, d > 0.  
\end{equation}
\end{theorem}

\section{Conjugate gradient method}\label{sec4}
In this section, we employ a conjugate gradient method to approximate the minimizer of the functional $J(f_1, f_2)$ defined by \eqref{s2111_1}. Let $(f_1^k, f_2^k)$ be the $k$th approximation of the solution $(f_1, f_2)$. The iterations are given by the following updated equations:
\begin{equation}\label{process}
f_1^{k+1}=f_1^k+\zeta_1^k S_1^k, \quad f_2^{k+1}=f_2^k+\zeta_2^k S_2^k, \quad k=0,1,2, \ldots,
\end{equation}
here $f_1^0(x)$ and $f_2^0(x)$ are the initial guesses for $f_1(x)$ and $f_2(x)$, respectively. For $i=1,2,$ the terms $\zeta_i^k$ are the step sizes, and $S_i^k$ represent the search directions which are determined by the following equations:
\begin{equation}\label{direction}
S_i^k=\left\{\begin{array}{ll}
-\tilde{J}_i^{\prime 0}, & k=0, \\
-\tilde{J}_i^{\prime k}+\vartheta_i^k\, S_i^{k-1}, & k \geq 1,
\end{array} \right.
\end{equation}
for $i=1,2,$ where $\tilde{J}_i^{\prime k}:=\tilde{J}_{f_i}^{\prime}\left(f_1^k, f_2^k\right)$ denotes the derivative of $\tilde{J}$ with respect to $f_i$, and $\vartheta_i^k$ are the conjugate coefficients obtained using the Fletcher-Reeves formula \cite{salah1}
\begin{equation}\label{conj_coef}
\vartheta_i^k=\frac{\left\|\tilde{J}_i^{\prime k}\right\|_{L_2(\Gamma_i^{\mathcal{T}})}}{\left\|\tilde{J}_i^{\prime k-1}\right\|_{L_2(\Gamma_i^{\mathcal{T}})}}, \quad k \geq 1.
\end{equation}
To determine the step sizes $\zeta_i^k$ $(i=1,2)$, we employ the line search method. Specifically, we solve the minimization problem given by
\begin{equation}\label{minstep}
\tilde{J}\left(f_1^{k+1}, f_2^{k+1}\right)=\min _{\zeta_1, \zeta_2 \geq 0} \tilde{J}\left(f_1^k+\zeta_1 S_1^k, f_2^k+\zeta_2 S_2^k\right).    
\end{equation}
\vspace{-0.2cm}
Using this formulation, we can derive the partial derivatives of $\tilde{J}\left(f_1^{k+1}, f_2^{k+1}\right)$ with respect to $\zeta_1^k$ as follows:
$$
\begin{aligned}
\frac{\partial \tilde{J}\left(f_1^{k+1}, f_2^{k+1}\right)}{\partial \zeta_1^k} & =\lim _{\zeta_1^k \rightarrow 0} \frac{\tilde{J}\left(f_1^{k+1}, f_2^{k+1}\right)-\tilde{J}\left(f_1^{k}, f_2^{k+1}\right)}{\zeta_1^k} \\
&=\lim _{\zeta_1^k \rightarrow 0}\frac{\zeta_1^k\,\int_0^\mathcal T \int_{\Gamma_1} S_1^k\, \varphi_1^{k+1}\, dx\, dt + o(\left\|\zeta_1^k\, S_1^k\right\|)}{\zeta_1^k}  \\
& =\int_0^\mathcal{T}\int_{\Gamma_1} \tilde{J}_1^{\prime k+1}\, S_1^k(x,t)\,  dx\, dt.
\end{aligned}
$$
Similarly, we can derive the partial derivative with respect to $\zeta_2^k$ as
$$
\frac{\partial \tilde{J}\left(f_1^{k+1}, f_2^{k+1}\right)}{\partial \zeta_2^k}=\int_0^\mathcal{T}\int_{\Gamma_2} \tilde{J}_2^{\prime k+1}\, S_2^k(x,t)\,  dx\, dt.
$$
Thus, the minimization problem \eqref{minstep} implies that the step sizes $\zeta_1^k$ and $\zeta_2^k$ satisfy the following conditions:
\begin{equation}\label{first_cond}
   \int_0^\mathcal{T}\int_{\Gamma_1} \tilde{J}_1^{\prime k+1}\, S_1^k(x,t)\,  dx\, dt=0, \quad  \int_0^\mathcal{T}\int_{\Gamma_2} \tilde{J}_2^{\prime k+1}\, S_2^k(x,t)\,  dx\, dt=0, \quad k \geq 0. 
\end{equation}
\subsection{Convergence}
This subsection is devoted to the convergence of the iterative process \eqref{process}-\eqref{conj_coef}. The following theorem can be easily proved by following the arguments in \cite{lesnic} closely.  For the convenience of the reader, we present the main steps in the proof.
\begin{theorem}
Let $\{f_1^k, f_2^k\}_{k \geq 0} \in W$ be the iterations defined by the process described in equation \eqref{process}, and let $\zeta_1^k$ and $\zeta_2^k$ satisfy the condition in the equation \eqref{minstep}. Then the CGM \eqref{process}-\eqref{conj_coef} either terminates at a stationary point or converges in the following sense:
$$
\liminf _{k \rightarrow \infty}\left\|\tilde{J}_1^{\prime k}\right\|_{L_2(\Gamma_1^\mathcal T)}=0, \quad \liminf _{k \rightarrow \infty}\left\|\tilde{J}_2^{\prime k}\right\|_{L_2(\Gamma_2^\mathcal T)}=0 .
$$    
\end{theorem}
\begin{proof}
Assume that $\liminf _{k \rightarrow \infty}\left\|\tilde{J}_i^{\prime k}\right\|_{L_2(\Gamma_i^\mathcal T)} \neq 0$, for $i=1, 2$. Then, there exists a constant $M>0$ and an integer $k_0>0$ such that $\left\|\tilde{J}_i^{\prime k}\right\|_{L_2(\Gamma_i^\mathcal T)} \geq M$, for $k \geq k_0$ and $i=1, 2$. Then, on the one hand, from \eqref{conj_coef} and \eqref{first_cond}, one can get $$\left\|S_i^k\right\|_{L_2(\Gamma_i^\mathcal T)}^2=\left\|\tilde{J}_i^{\prime k}\right\|_{L_2(\Gamma_i^\mathcal T)}^2+\frac{\left\|\tilde{J}_i^{\prime k}\right\|_{L_2(\Gamma_i^\mathcal T)}^4}{\left\|\tilde{J}_i^{k-1}\right\|_{L_2(\Gamma_i^\mathcal T)}^4}\left\|S_i^{k-1}\right\|_{L_2(\Gamma_i^\mathcal T)}^2,\quad \text{for}\quad k>k_0.$$ Dividing both sides by $\left\|\tilde{J}_i^{\prime k}\right\|_{L_2(\Gamma_i^\mathcal T)}^4$, we obtain
$$
\frac{\left\|S_i^k\right\|_{L_2(\Gamma_i^\mathcal T)}^2}{\left\|\tilde{J}_i^{\prime k}\right\|_{L_2(\Gamma_i^\mathcal T)}^4}=\frac{1}{\left\|\tilde{J}_i^{\prime k}\right\|_{L_2(\Gamma_i^\mathcal T)}^2}+\frac{\left\|S_i^{k-1}\right\|_{L_2(\Gamma_i^\mathcal T)}^2}{\left\|\tilde{J}_i^{\prime k-1}\right\|_{L_2(\Gamma_i^\mathcal T)}^4}=\sum_{m=k_0}^k \frac{1}{\left\|\tilde{J}_i^{\prime m}\right\|_{L_2(\Gamma_i^\mathcal T)}^2} \leq \frac{k-k_0+1}{M}, \quad k>k_0.
$$
Then,
\begin{equation}\label{one_hand}
 \sum_{k \geq 0} \frac{\left\|\tilde{J}_i^{\prime k}\right\|_{L_2(\Gamma_i^\mathcal T)}^4}{\left\|S_1^k\right\|_{L_2(\Gamma_i^\mathcal T)}^2} \geq \sum_{k>k_0} \frac{\left\|\tilde{J}_i^{\prime k}\right\|_{L_2(\Gamma_i^\mathcal T)}^4}{\left\|S_i^k\right\|_{L_2(\Gamma_i^\mathcal T)}^2} \geq M \sum_{k \geq 1} \frac{1}{k+1}=\infty.   
\end{equation}
On the other hand, by using the Lipschitz continuity of the gradients $\tilde{J}_1^{\prime}$ and $\tilde{J}_2^{\prime}$, and by following the arguments in \cite{salah1, salah2, salah3, salah4}, we can obtain that
\begin{equation}\label{second_hand}
\sum_{k \geq 0} \frac{\left\|\tilde{J}_i^{\prime k}\right\|_{L_2(\Gamma_i^\mathcal T)}^4}{\left\|S_i^k\right\|_{L_2(\Gamma_i^\mathcal T)}^2}<\infty, \quad \text{for} \quad i=1, 2.     
\end{equation}
Therefore, we conclude that the inequalities in \eqref{one_hand} are in contradiction with the ones in \eqref{second_hand}. The proof is complete.
\end{proof}
 \subsection{Iterative algorithm}
 This subsection is concerned with the main steps of the iterative algorithm that will be used in the next section to find an approximate solution to the minimization problem. Based on the above discussion, one can remark that all the parameters are expressed explicitly except for the search step sizes $\zeta_1^k$, for $i=1, 2.$ In order to determine these two parameters, we use the line search method. The parameters $\zeta_i^k,$ for $i=1, 2,$ can be obtained by minimizing
$$
\begin{aligned}
\tilde{J}\left(f_1^{k+1}, f_2^{k+1}\right)&=\frac{1}{2} \int_0^\mathcal T\int_{\Gamma_1} \Big(\tilde{u}\left(f_1^k+\zeta_1^k S_1^k, f_2^k+\zeta_2^k S_2^k\right)-h_1\Big)^2 d x\, dt \\
& +\frac{1}{2} \int_0^\mathcal T\int_{\Gamma_2} \Big(\tilde{u}\left(f_1^k+\zeta_1^k S_1^k, f_2^k+\zeta_2^k S_2^k\right)-h_2\Big)^2 d x\, dt.
\end{aligned}
$$
Since we have $$\tilde{u}\left(f_1^k+\zeta_1^k S_1^k, f_2^k+\zeta_2^k S_2^k\right):=\tilde{u}\left(f_1^k, f_2^k\right)+\zeta_1^k\, \tilde{u}^0\left(S_1^k, 0\right)+\zeta_2^k\,\tilde{u}^0\left(0, S_2^k\right),$$
where $\tilde{u}^0$ is the solution to the forward problem \eqref{s3-2} with source term $F\equiv 0$,  the function $\tilde{J}\left(f_1^{k+1}, f_2^{k+1}\right)$ to be minimized can be rewritten as follows:
$$\tilde{J}\left(f_1^{k+1}, f_2^{k+1}\right)=\frac{1}{2} \sum_{i=1}^2\int_0^\mathcal T\int_{\Gamma_i}\Big(\tilde{u}\left(f_1^k, f_2^k\right)+\zeta_1^k\, \tilde{u}^0\left(S_1^k, 0\right)+\zeta_2^k\,\tilde{u}^0\left(0, S_2^k\right)-h_i\Big)^2 d x\, dt.$$
It follows
$$
\frac{\partial \tilde{J}\left(f_1^{k+1}, f_2^{k+1}\right)}{\partial \zeta_1^k}=\mathcal R_1 \zeta_1^k+\mathcal R_2 \zeta_2^k+\mathcal R_3, \quad \frac{\partial \tilde{J}\left(f_1^{k+1}, f_2^{k+1}\right)}{\partial \zeta_2^k}=\mathcal R_2 \zeta_1^k+\mathcal R_4 \zeta_2^k+\mathcal R_5,
$$
where
$$\mathcal R_1=\sum_{i=1}^2\int_0^\mathcal T\int_{\Gamma_i} \Big|\tilde{u}^0(S_1^k,0)\Big|^2\, dx\, dt, \quad \mathcal R_2=\sum_{i=1}^2\int_0^\mathcal T\int_{\Gamma_i} \tilde{u}^0(S_1^k,0)\, \tilde{u}^0(0,S_2^k)\, dx\, dt, $$ $$\mathcal R_3=\sum_{i=1}^2\int_0^\mathcal T\int_{\Gamma_i} \tilde{u}^0(S_1^k,0)\,\Big(\tilde{u}(f_1^k,f_2^k)-h_i\Big)\, dx\, dt, \quad \mathcal R_4=\sum_{i=1}^2\int_0^\mathcal T\int_{\Gamma_i} \Big|\tilde{u}^0(0,S_2^k)\Big|^2\, dx\, dt,\quad $$
$$\text{and}\quad \mathcal R_5=\sum_{i=1}^2\int_0^\mathcal T\int_{\Gamma_i} \tilde{u}^0(0,S_2^k)\,\Big( \tilde{u}(f_1^k,f_2^k)-h_i\Big)\, dx\, dt.$$
By setting $
\frac{\partial \tilde{J}\left(f_1^{k+1}, f_2^{k+1}\right)}{\partial \zeta_1^k}=0$ and $\frac{\partial \tilde{J}\left(f_1^{k+1}, f_2^{k+1}\right)}{\partial \zeta_2^k}=0,$ one can deduce that the search step sizes $\zeta_1^k$ and $\zeta_2^k$ satisfy 
$$
\mathcal R_1 \zeta_1^k+\mathcal R_2 \zeta_2^k+\mathcal R_3 = 0,\, \, \, \text{ and }\; 
\mathcal R_2 \zeta_1^k+\mathcal R_4 \zeta_2^k+\mathcal R_5=0.
 $$
Therefore, the search step sizes $\zeta_1^k$ and $\zeta_2^k$ are given as follows:
\begin{equation}\label{search_step}
\zeta_1^k=\frac{\mathcal R_3 \mathcal R_4-\mathcal R_2 \mathcal R_5}{\mathcal R_2^2-\mathcal R_1 \mathcal R_4}, \quad \zeta_2^k=\frac{\mathcal R_1 \mathcal R_5-\mathcal R_2 \mathcal R_3}{\mathcal R_2^2-\mathcal R_1 \mathcal R_4} .
\end{equation}
It is well known that the most important step is to find a suitable stopping rule in an iteration algorithm. To deal with this issue, in the current work, we apply the discrepancy principle to stop the iteration procedure. According to the discrepancy principle, the iterative procedure is stopped when the following criterion is satisfied:
\begin{equation}\label{stop}
\tilde{J}\left(f_1^k, f_2^k\right) \leq \bar{\epsilon},
\end{equation}
with $\bar{\epsilon}=\frac{1}{2} \sum_{i=1,2}\left\|h_i^\epsilon-h_i\right\|_{L_2(\Gamma_i^\mathcal{T})}^2 \leq \epsilon^2$, where $h_i^\epsilon$, for $i=1, 2,$ is noisy perturbation of the data $h_i$. To this end, we summarize the main steps of our reconstruction approach in the following algorithm: 

\vspace{0.2cm}
\begin{algorithm}[H]
\begin{enumerate}

\vspace{0.2cm}
\item[{\bf Step 1.}] Set $k=0$ and choose initial guesses $f_1^0$ and $f_2^0$ for the unknowns $f_1$ and $f_2$, respectively.
\item[{\bf Step 2.}] Solve the direct problem \eqref{s3-2} with $(f_1, f_2)=(f_1^k, f_2^k)$ to get $\tilde{u}(f_1^k, f_2^k)$, and compute  the residuals 
$$r_{i}^k=\tilde{u}(f_1^k, f_2^k)\big|_{\Gamma_i^\mathcal{T}}-h_i^\epsilon, \quad \text{for}\, \, i=1, 2.$$
\item[{\bf Step 3.}] Solve the adjoint problem \eqref{s3-7} and evaluate the gradients $\tilde{J}_1^{\, \prime}(f_1^k, f_2^k)$ and $\tilde{J}_2^{\, \prime}(f_1^k, f_2^k)$.
\item[{\bf Step 4.}] Calculate the conjugate coefficients $\gamma_{i}^k$ by \eqref{conj_coef} and the directions $S_{i}^k$ by \eqref{direction}, for $i=1, 2$.
\item[{\bf Step 5.}] Calculate the step sizes $\zeta_{1}^k$ and $\zeta_{2}^k$ by \eqref{search_step}.
\item[{\bf Step 6.}] Update $(f_1^{k+1}, f_2^{k+1})$ by \eqref{process}.
\item[{\bf Step 7.}] If the condition \eqref{stop} is satisfied, then go to {\bf Step 8}. Otherwise set $k = k + 1$ and go to {\bf Step 2}.
\item[{\bf Step 8.}] End.

\end{enumerate}
    \caption{\it Conjugate Gradient Method {\bf (CGM)}}
    \label{algo1}
\end{algorithm}
\section{Numerical experiments}\label{sec5}
This section deals with the numerical solutions for both the direct and inverse problems. We first present the methods used for solving the forward and the adjoint problems. Then, we employ the Conjugate Gradient Method (CGM) for two-dimensional numerical computations with the aim of simultaneously reconstructing fluxes $f_1$ and $f_2$. We present numerical results from three example scenarios, highlighting the effectiveness of the CGM. The section also investigates the algorithm's convergence and stability, evaluating its overall effectiveness and performance in producing precise reconstructions.
\subsection{Numerical solutions of the forward and adjoint problems}\label{sol-iter}
In this subsection, we introduce an iterative approach for solving the direct and adjoint problems \eqref{s3-2} and \eqref{s3-7} numerically. To begin with, we consider a general fractional model as follows:

\begin{equation*}
(\mathcal{P}_\chi)\left\{\begin{aligned}
D_\chi^\beta u-div(k(|\nabla u|^2)\, \nabla u) & =F &  \text{in}\quad & \Omega_{\mathcal{T}}, \\
-k(|\nabla u|^2) \frac{\partial u} {\partial n}&=f_1 &\text{on}\quad & \Gamma_1^\mathcal{T},
\\
-k(|\nabla u|^2) \frac{\partial u} {\partial n}&=f_2 &\text{on} \quad& \Gamma_2^\mathcal{T},
\\
u & =0 &\text{on} \quad& \Gamma_3^\mathcal T \cup \Gamma_4^\mathcal T ,\\
u & =g &\text{in}\quad &   \Omega \times \{\chi\},
\end{aligned}\right. 
\end{equation*}
where $\chi = 0, \mathcal{T}$, the fractional operator $D_\chi^\beta$ represents the left Caputo fractional derivative for $\chi = 0$ and the right Caputo fractional derivative for $\chi = \mathcal{T}$.

Since the boundary conditions in problems $(\mathcal{P}_0)$ and $(\mathcal{P}_\mathcal T)$ are implicitly defined, finding their numerical solutions becomes challenging. Specifically, the explicit numerical approximation at interior points becomes unfeasible without knowledge of the solution values at the boundaries. To address this issue, we rely on that the solutions of problems $(\mathcal{P}_0)$ and $(\mathcal{P}_\mathcal T)$ can be obtained as limits of the solutions to the linearized problems \eqref{s3-2} and \eqref{s3-7}, respectively. With this in mind, we aim to develop an iterative approach that solves the following system:
\begin{equation*}
(\mathcal{LP}_\chi)\left\{\begin{aligned}
D_\chi^\beta u-div(\mathcal K\, \nabla u) & =F &  \text{in}\quad & \Omega_\mathcal T, \\
-\mathcal K(x,y,t)\, \frac{\partial u} {\partial \mathbf{n}}&=f_1 &\text{on}\quad & \Gamma_1^\mathcal T,
\\
-\mathcal K(x,y,t)\, \frac{\partial u} {\partial \mathbf{n}}&=f_2 &\text{on} \quad& \Gamma_2^\mathcal T,
\\
u & =0 &\text{on} \quad& \Gamma_3^\mathcal T \cup \Gamma_4^\mathcal T ,\\
u & =g &\text{in}\quad &   \Omega \times \{\chi\},
\end{aligned}\right. 
\end{equation*}
where the function $\mathcal{K}(x,y,t)$ is given appropriately.
\subsubsection{Iterative algorithm for solving forward and adjoint problems}
In this subsection, we aim to provide an efficient numerical approach that can accurately approximate the solutions of problems $(\mathcal P_\chi)$, for $\chi=0, \mathcal T$. Subsequently, we outline the iterative algorithm, which iteratively refines the solution and ensures convergence toward the desired accuracy. The algorithm's stepwise progression and convergence criteria are presented in detail, enabling the efficient numerical solution of the direct and adjoint problems. \\

As previously mentioned, our proposed method solves the linearized problems $(\mathcal{LP}_0)$ and $(\mathcal{LP}_{\mathcal T})$ numerically. To achieve this, we adopt a methodology akin to that found in \cite{salah6}, employing the implicit finite difference method (IFDM) as our primary tool for numerically solving partial differential equations. We arrive now to present the iterative approach to approximate the solution for problem $(\mathcal{P}_\chi)$. The iteration process is terminated when the following criterion is fulfilled:
\begin{equation}\label{stop2}
  \|u^{(n+1)}-u^{(n)}\|_{L_2(0,\mathcal T; H_0^1(\Omega))} \leq \bar{\theta},
\end{equation}
where $\bar{\theta} \ll 1$ is a tolerance parameter, $u^{(0)}\equiv 0$, and for $n \geq 1$, $u^{(n)}$ denotes the solution to the linearized problem associated to $(\mathcal{P}_\chi)$ which will be denoted hereafter by $(\mathcal{P}_\chi^{(n)})$. From \eqref{convergen}, the approximated solution $u_{app}$ of the system $(\mathcal{P}_\chi)$ can be taken as the limit of the sequence $\{ u^{(n)}\}_{n\geq 0}$ in the $L_2(0,\mathcal T; H_0^1(\Omega))$ norm. Namely, the approximated solution $u_{app}$ can be taken as
$$u_{app} = u^{(\eta^\star)}\; \text{ where }\; \eta^\star:= \min_{n\in \mathbb{N}}\big\{n; \text{ such that }\;  \|u^{(n+1)}-u^{(n)}\|_{L_2(0,\mathcal T; H_0^1(\Omega))}\leq \bar{\theta}\big \}.$$
In summary, the main steps of the proposed approach are outlined in the following algorithm:

\begin{algorithm}[H]
\begin{enumerate}

\vspace{0.2cm}
\item[{\bf Step 1}] Choose a tolerance parameter $\bar{\theta}$ and set $n = 0$. Initialize $u^{(0)}$ as $0$.

\item[{\bf Step 2}] Update the solution by solving problem $(\mathcal{P}_\chi^{(n+1)})$ using the current approximation $u^{(n)}$ to obtain $u^{(n+1)}$.

\item[{\bf Step 3}] Compute $\|u^{(n+1)}-u^{(n)}\|_{l_2(0,\mathcal{T}; H_0^1(\Omega))}$. If condition \eqref{stop2} is satisfied go to {\bf Step 5}. Otherwise, proceed to the next step.

\item[{\bf Step 4}] Increment the iteration counter: $n \rightarrow n + 1$. Return to {\bf Step 2}.

\item[{\bf Step 5}] Output the final approximated solution $u^{(\eta^\star)}$ and the stopping index $\eta^\star$ as the final results.

\end{enumerate}
    \caption{\it Iterative Approach for Solving Problems \eqref{s3-2} and \eqref{s3-7}}
    \label{solving-direct}
\end{algorithm}

By following this algorithm, we can iteratively improve the approximation of the solution for problem $(\mathcal{P}_\chi)$, ensuring convergence towards the desired accuracy.
\begin{remarks}\quad\\
\vspace{-0.9cm}
\begin{itemize}
   \item In case the exact solution is known or available, the stopping criterion \eqref{stop2} in the algorithm can be replaced by an alternative criterion. One such alternative is to use the condition 
   $\|u_{ex}-u^{(n)}\|_{L_2(0,\mathcal T; H_0^1(\Omega))} \leq \bar{\theta},$ which measures the discrepancy between the exact solution $u_{ex}$ and the solution $u^{(n)}$ at each iteration.

    \item The choice of the tolerance parameter $\bar{\theta}$ is crucial in balancing solution accuracy and computational efficiency. Smaller $\bar{\theta}$ values yield higher accuracy but may increase convergence time due to more iterations. Striking the right balance between accuracy and efficiency is vital when selecting $\bar{\theta}$.
    \end{itemize}
\end{remarks}
\subsubsection{Validation of the proposed algorithm}
In this subsection, we provide two examples to evaluate the performance of our numerical approach. Specifically, we test the iterative Algorithm  \ref{solving-direct} outlined in the previous subsection to solve the problems $(\mathcal{P}_0)$ and $(\mathcal{P}_\mathcal T)$ with known analytical solutions. Without loss of generality, in our computational procedure, we defined the computational domain $\Omega$ as the square region $[0, 1] \times [0, 1]$ and set the final time $\mathcal{T}$ to be 1. The step time and mesh sizes are taken to be $\tau=0.001$ and $h=0.05$, respectively.

\paragraph{\bf Example 1.} In this example, we  consider the problem $(\mathcal{P}_0)$ with a known exact analytical solution. The exact solution is given by:
$$u_{ex}(x,y,t)=E_{\beta,1}(-t^\beta)\,(1-x)\,(1-y),\; \; \forall (x,y,t) \in \Omega \times [0, \mathcal{T}].$$
By substituting this exact solution into the system $(\mathcal{P}_0)$, we can determine the corresponding forcing term, initial value, and fluxes. After obtaining the necessary inputs for the problem $(\mathcal{P}_0)$ based on the known exact solution, we proceed to apply Algorithm  \ref{solving-direct} to solve the problem numerically. In the first step, we examine the convergence behavior of the iteration process by analyzing the variation of the quantity $\|u^{(n+1)}-u^{(n)}\|_{L_2(0,\mathcal T; H_0^1(\Omega))}$ with respect to the iteration index $n$. Figure \ref{con_his_direct} illustrates this variation for two values of the parameter $\beta$, namely $\beta=0.3$ and $\beta=0.7$.
\begin{figure}[H]
     \centering
     \includegraphics[height=1.8in]{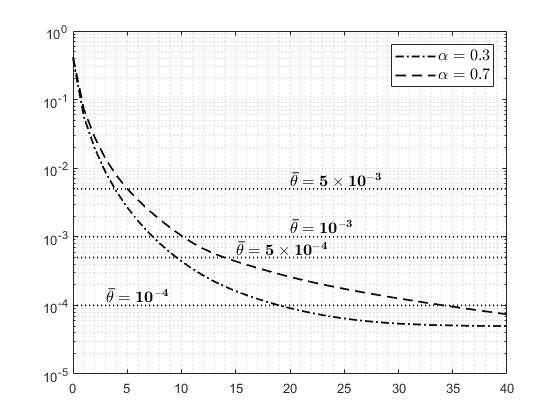}
         \caption{Convergence behavior of the iteration process: $\|u^{(n+1)}-u^{(n)}\|_{L_2(0,\mathcal T; H_0^1(\Omega))}$ with iteration index $n$.}
        \label{con_his_direct}
\end{figure}
Next, we present a comprehensive analysis of the algorithm's performance with respect to the different choices of the tolerance parameter $\bar{\theta}$ indicated in Figure \ref{con_his_direct}. For each chosen value of $\bar{\theta}$, the corresponding stopping index $\eta^\star$, the CPU time required for convergence, and the error $\|u_{ex}-u^{(\eta^\star)}\|{L_2(0,\mathcal T; H_0^1(\Omega))}$ are given in Table \ref{errors}.

\begin{table}[H]
    \centering
    \small
    \begin{tabular}{|l|c|c|c|c|c|}
        \hline
        & $\bar{\theta}$ & $5\times 10^{-3}$ & $10^{-3}$ & $5\times 10^{-4}$ & $10^{-4}$ \\
        \hline 
        \multirow{3}{*}{\centering $\beta=0.3$} & $\eta^\star$ & $4$&  $7$& $9$& $19$ \\
        \cline{2-6}
         & $\|u_{ex}-u^{(\eta^\star)}\|_{L_2(0,\mathcal T; H_0^1(\Omega))}$ & $1.23e-02$ & $4.91e-03$&$3.73e-03 $&$2.84e-03$ \\
        \cline{2-6}
         & CPU (sec) & $4.82$&$8.11$&$10.61$&$21.61$ \\
        \hline
        \multirow{3}{*}{\centering $\beta=0.7$} & $\eta^\star$ & $5$&$10$&$14$&$34$ \\
        \cline{2-6}
         & $\|u_{ex}-u^{(\eta^\star)}\|_{L_2(0,\mathcal T; H_0^1(\Omega))}$ & $2.11e-02$&$8.47e-03$&$5.83e-03$&$3.31e-03$ \\
        \cline{2-6}
         & CPU (sec) & $5.85$&$11.01$&$15.84$&$42.87$ \\
        \hline 
    \end{tabular}
    \caption{Performance for different tolerance parameters $\bar{\theta}$.}
    \label{errors}
\end{table}
In order to assess the accuracy and convergence of our approximation method, we compare the exact solution $u_{ex}$, and the solution obtained using our algorithm $u^{(\eta^\star)}$ at time $t=0.5$, for $\beta=0.3$ and $\beta=0.7$. Figure \ref{test_direct} provides a visual representation of this comparison, where the left subplot shows the exact solution $u_{ex}$, the middle subplot displays the solution obtained using our algorithm $u^{(\eta^\star)}$, and the right subplot presents the absolute error $|u_{ex}-u^{(\eta^\star)}|$ between the two solutions.
\begin{figure}[H]
     \centering
     \begin{subfigure}[b]{0.3\textwidth}
         \centering
         \includegraphics[height=1.7in]{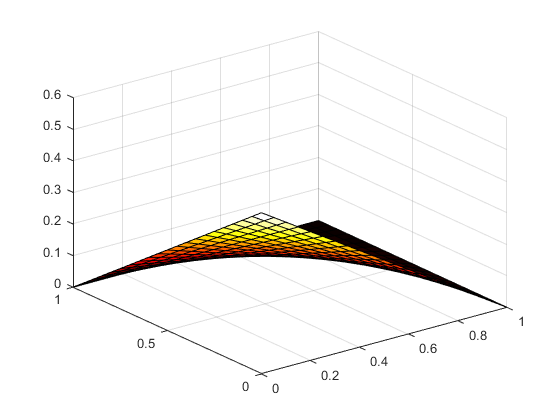}
         \caption{Exact solution}
         \label{fig:y equals x}
     \end{subfigure}
     \hfill
     \begin{subfigure}[b]{0.3\textwidth}
         \centering
         \includegraphics[height=1.7in]{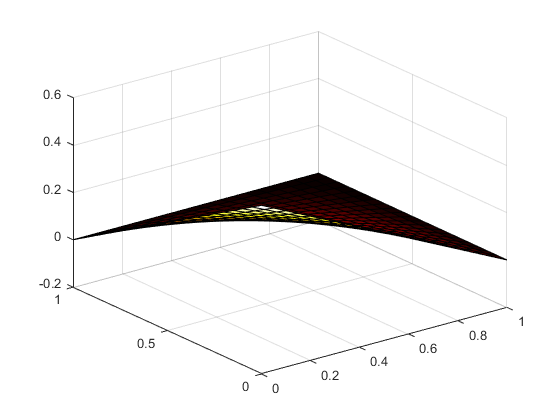}
         \caption{Approximated solution}
         \label{fig:three sin x}
     \end{subfigure}
     \hfill
     \begin{subfigure}[b]{0.3\textwidth}
         \centering
         \includegraphics[height=1.7in]{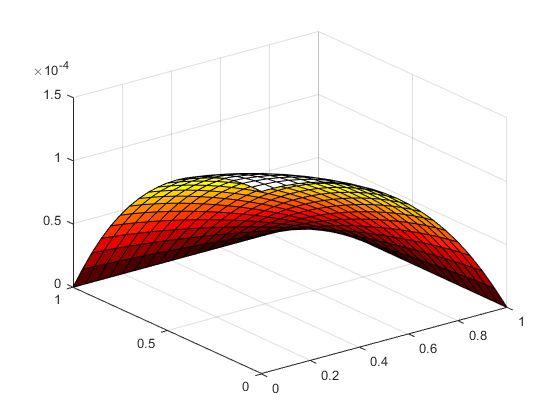}
         \caption{Absolute error}
         \label{fig:five over x}
     \end{subfigure}\\
     \begin{subfigure}[b]{0.3\textwidth}
         \centering
         \includegraphics[height=1.7in]{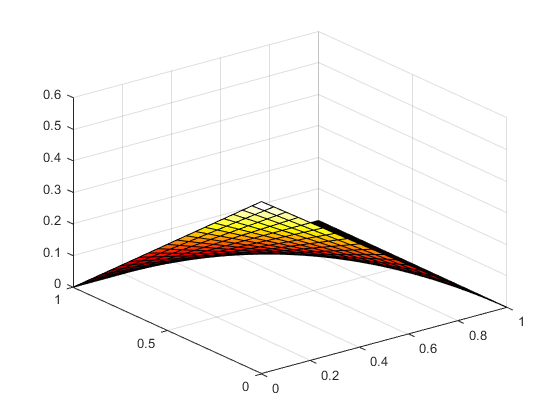}
         \caption{Exact solution}
         \label{fig:y equals x}
     \end{subfigure}
     \hfill
     \begin{subfigure}[b]{0.3\textwidth}
         \centering
         \includegraphics[height=1.7in]{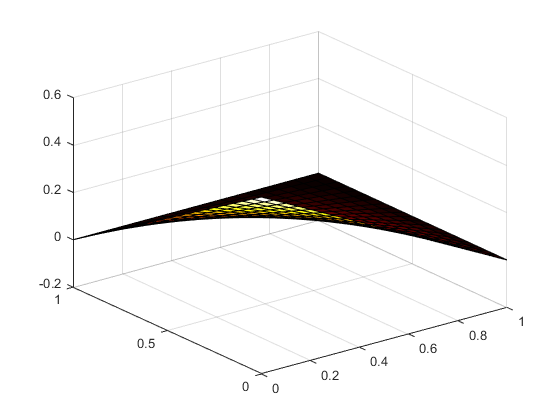}
         \caption{Approximated solution}
         \label{fig:three sin x}
     \end{subfigure}
     \hfill
     \begin{subfigure}[b]{0.3\textwidth}
         \centering
         \includegraphics[height=1.7in]{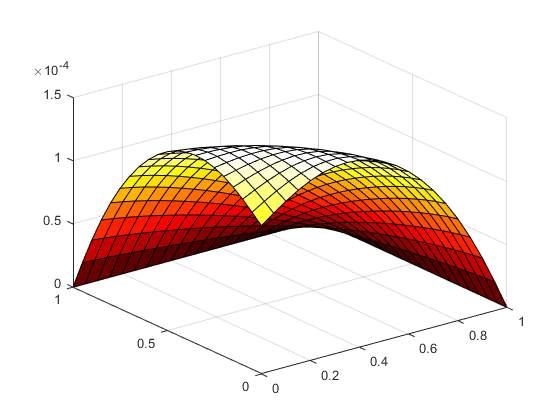}
         \caption{Approximated solution}
         \label{fig:five over x}
     \end{subfigure}
        \caption{Obtained results for Example 1 at time $t=0.5$. The top row represents the obtained results for $\beta=0.3$. The bottom row represents the obtained results for $\beta=0.7$. In each row; left: the analytic solution, middle: the approximated solution, right: the absolute error.}
        \label{test_direct}
\end{figure}
\paragraph{\bf Example 2} 
In this example, we aim to assess the performance of Algorithm \ref{solving-direct} in solving the problem $(\mathcal{P}_\mathcal T)$ which is an essential component in finding the solution to the optimization problem. We utilize the same algorithm as in the previous test, with specific modifications for the problem $(\mathcal{P}_\mathcal T)$. The exact solution in this example is given by $$v_{ex}(x,y,t)=(T-t)^{2\beta}\, (1-x)\, (1-y),\; \; \forall (x,y,t) \in \Omega \times [0, \mathcal{T}].$$ 
By substituting this exact solution $v_{ex}(x,y,t)$ into the system $(\mathcal{P}_\mathcal T)$, we can easily get the corresponding forcing term, the final value, and the fluxes. Similar to the first test, we investigate the convergence behavior of the iteration process by analyzing the variation of the quantity $\|v^{(n+1)}-v^{(n)}\|_{L_2(0,\mathcal T; H_0^{1}(\Omega))}$ with respect to the iteration index $n$. Figure \ref{con_his_adj} depicts the convergence history for two different values of the parameter $\beta$: $\beta=0.3$ and $\beta=0.7$.
\begin{figure}[H]
     \centering
     \includegraphics[height=1.8in]{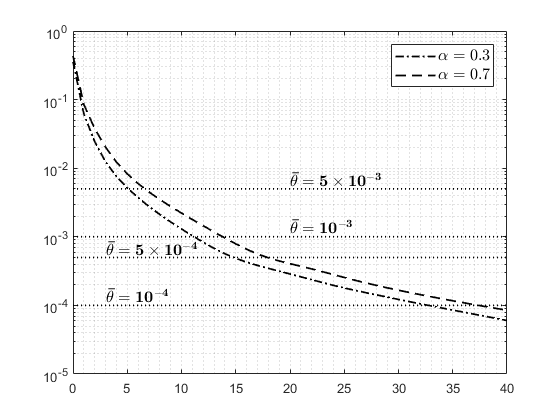}
         \caption{Convergence behavior of the iteration process: $\|v^{(n+1)}-v^{(n)}\|_{L_2(0,\mathcal T; H_0^1(\Omega))}$ with iteration index $n$.}
        \label{con_his_adj}
\end{figure}
In Table \ref{errors_adj}, we provide a detailed overview of the algorithm's performance for different values of the tolerance parameter $\bar{\theta}$. Specifically, we present the stopping index $\eta^\star$, which indicates the iteration at which the algorithm achieves convergence, the CPU time required for the algorithm to reach convergence, and the error $\|v_{ex}-v^{(\eta^\star)}\|_{L_2(0,\mathcal T; H_0^{1}(\Omega))}$, which quantifies the dissimilarity between the exact solution $v_{ex}$ and the final approximated one $v^{(\eta^\star)}$.
\begin{table}[H]
    \centering
    \small
    \begin{tabular}{|l|c|c|c|c|c|}
        \hline
        & $\bar{\theta}$ & $5\times 10^{-3}$ & $10^{-3}$ & $5\times 10^{-4}$ & $10^{-4}$ \\
        \hline 
        \multirow{3}{*}{\centering $\beta=0.3$} & $\eta^\star$ & $5$ & $11$ & $15$ & $33$ \\
        \cline{2-6}
         & $\|v_{ex}-v^{(\eta^\star)}\|_{L_2(0,\mathcal T; H_0^1(\Omega))}$ & $2.63e-02$ & $1.18e-02$ & $8.91e-03$ & $4.62e-03$ \\
        \cline{2-6}
         & CPU (sec) & $5.83$ & $12.03$ & $16.63$ & $40.87$ \\
        \hline
        \multirow{3}{*}{\centering $\beta=0.7$} & $\eta^\star$ & $7$ & $14$ & $17$ & $37$ \\
        \cline{2-6}
         & $\|v_{ex}-v^{(\eta^\star)}\|_{L_2(0,\mathcal T; H_0^1(\Omega))}$ & $3.81e-02$ & $2.28e-02$ & $1.69e-02$ & $9.83e-03$ \\
        \cline{2-6}
         & CPU (sec) & $8.13$ & $15.87$ & $18.93$ & $54.22$ \\
        \hline 
    \end{tabular}
    \caption{Performance for different tolerance parameters $\bar{\theta}$.}
    \label{errors_adj}
\end{table}
To evaluate the accuracy and convergence of our approximation method for the adjoint problem, we illustrate in Figure \ref{test_adj} a comparison between the exact adjoint solution $v_{ex}$ with the solution obtained using our algorithm $v^{(\eta^\star)}$ at $t=0.5$ for $\beta=0.3$ and $\beta=0.7.$
\begin{figure}[H]
     \centering
     \begin{subfigure}[b]{0.3\textwidth}
         \centering
         \includegraphics[height=1.7in]{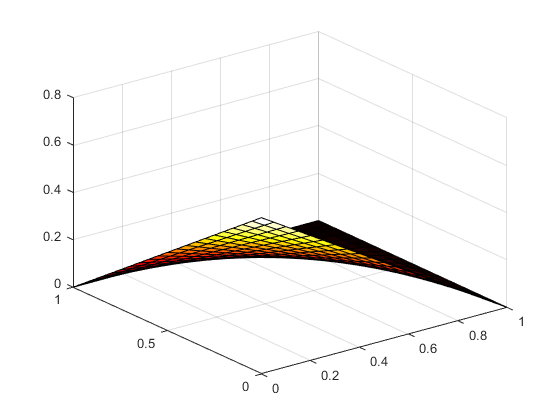}
         \caption{Exact solution}
         \label{fig:y equals x}
     \end{subfigure}
     \hfill
     \begin{subfigure}[b]{0.3\textwidth}
         \centering
         \includegraphics[height=1.7in]{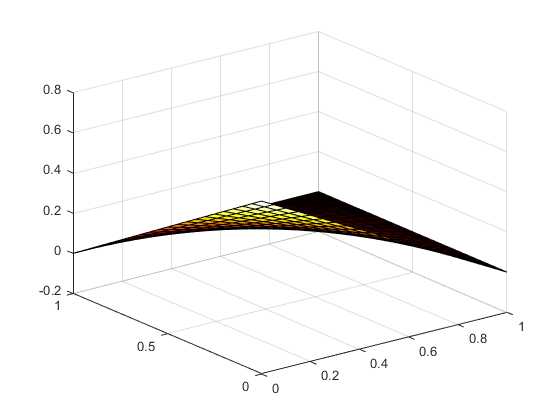}
         \caption{Approximated solution}
         \label{fig:three sin x}
     \end{subfigure}
     \hfill
     \begin{subfigure}[b]{0.3\textwidth}
         \centering
         \includegraphics[height=1.7in]{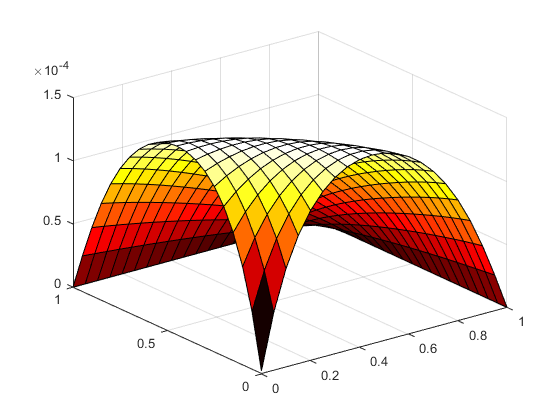}
         \caption{Absolute error}
         \label{fig:five over x}
     \end{subfigure}\\
     \begin{subfigure}[b]{0.3\textwidth}
         \centering
         \includegraphics[height=1.7in]{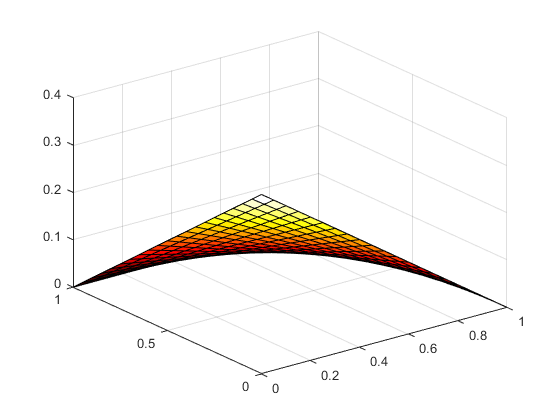}
         \caption{Exact solution}
         \label{fig:y equals x}
     \end{subfigure}
     \hfill
     \begin{subfigure}[b]{0.3\textwidth}
         \centering
         \includegraphics[height=1.7in]{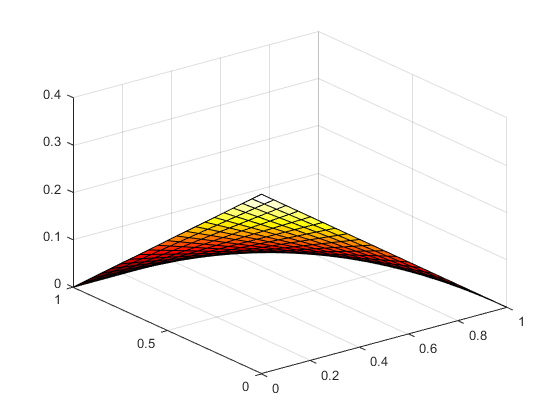}
         \caption{Approximated solution}
         \label{fig:three sin x}
     \end{subfigure}
     \hfill
     \begin{subfigure}[b]{0.3\textwidth}
         \centering
         \includegraphics[height=1.7in]{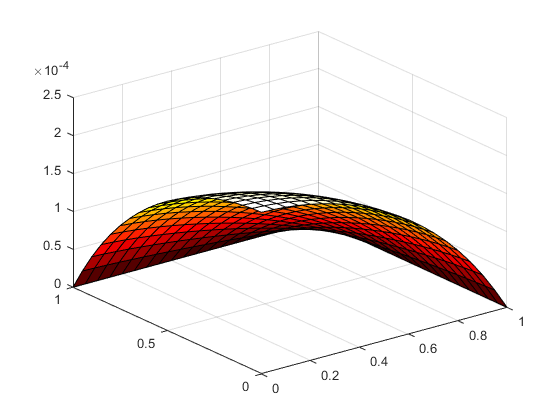}
         \caption{Absolute error}
         \label{fig:five over x}
     \end{subfigure}
        \caption{Obtained results for Example 2 at time $t=0.5$. The top row represents the obtained results for $\beta=0.3$. The bottom row represents the obtained results for $\beta=0.7$. In each row; left: the analytic solution, middle: the approximated solution, right: the absolute error.}
        \label{test_adj}
\end{figure}
In conclusion, Algorithm \ref{solving-direct} has proven effective in approximating solutions to problem $(\mathcal{P}_\chi)$. Our iterative computations and convergence analysis show that the algorithm reliably converges to the desired accuracy, as seen in the close match between the exact and algorithm-obtained solutions. The small absolute error further confirms the algorithm's reliability and accuracy in approximating the solution of the forward problem. These results highlight the algorithm's promise for solving the given problems and its potential utility in inversion algorithms.
\subsection{Numerical solution of the inverse problem}
This section is concerned with the numerical treatment of the minimization problem \eqref{s21111_1}. The goal is to identify the fluxes, denoted as $f_1$ and $f_2$, in the initial-boundary value problem. To achieve this, we will utilize Algorithm \ref{algo1} to obtain an approximation of the minimizer for the cost functional. To generate noisy observations $h_i^\epsilon$ (for $i=1, 2$), we simulate them by adding Gaussian additive noise to their analytical values, if available. The noisy observations are given by:
\begin{equation}\label{noiss}
  h_i^\epsilon=h_i+\gamma\ \textit{randn}(size(h_i))\,\|h_i\|_{L_2(0,\mathcal T; L_2(\Gamma_i))},  
\end{equation}
where $\gamma \geq 0$ represents the noise level and $\textit{randn}(.)$ is a "MATLAB" function that generates arrays of normally distributed random numbers with mean $0$ and standard deviation $\sigma = 1$. The corresponding noise level is calculated as:
$$\epsilon = \|h_i^\epsilon - h_i\|_{L_2(0,\mathcal T; L_2(\Gamma_i))}.$$
To show the accuracy of the numerical solutions, we compute the approximate $L_2$ error
denoted by
$$E(k, \gamma, f_i)=\|f_i^{k} - f_i^{ex}\|_{L_2(0,\mathcal T; L_2(\Gamma_i))}, \, \text{ for }\, i=1, 2,$$
where $f_i^{k}$ is the flux reconstructed at the $k$th iteration, and $f_i^{ex}$ is the exact solution.

\noindent To begin with, we assign in the numerical calculations the parameters and their values as follows:
\begin{itemize}
    \item The computational domain $\Omega$ is taken to be $\Omega=[0, 1] \times [0, 1]$.
    \item The final time $\mathcal T$ is fixed at $\mathcal T=1$.
    \item Since the algorithm involves multiple loops and has a significant computational time, the execution process will be expedited by choosing the step sizes in the IFDM as $h_x=0.05$ and $\tau=0.01$.
\item  For solving the forward and adjoint problems by Algorithm \ref{solving-direct}, the stopping index $\eta^\star$ is taken to be $\eta^\star=20.$
 \item The initial guesses $f_1^0$ and $f_2^0$ are both zero functions, i.e., $f_1^0=f_2^0 \equiv 0$.
 \item When the measured data is noise-free (i.e. $\gamma =0\%$), the parameter $\bar{\epsilon}$ is taken to be $1.25\times 10^{-7}$.
\end{itemize}
Hereafter, we denote by $k^\star$ the stopping index for our algorithm, which represents the minimum value such that condition \eqref{stop} is satisfied. In other words, 
$$k^\star:=\min\{k \in \mathbb{N}; \text{ such that } \tilde{J}\left(f_1^k, f_2^k\right) \leq \bar{\epsilon}\}.$$
With these considerations in place, we can now apply the conjugate gradient algorithm (Algorithm \ref{algo1}) to solve the problem \eqref{s21111_1} and present numerical results for three examples. This will demonstrate the effectiveness of the conjugate gradient algorithm, along with investigating its convergence and stability.
\paragraph{Example 1.} Let the exact solution for problem \eqref{s2-1} be $u(x,y,t)=t^\beta\, \log(2-x)\, (1-y)$, then 
\begin{itemize}
    \item $f_1(y,t)=\frac{2\,t^\beta\,(y-1)}{4+t^{2\alpha}\,((y-1)^2+4\log^2(2))},\, \, \forall (y,t) \in (0,1) \times (0,\mathcal T),$
    \item $f_2(x,t)=\frac{2\,t^\beta\,(2-x)^2\,\log(2-x)}{(2-x)^2+t^{2\beta}+t^{2\beta}\,(2-x)^2\,\log^2(2-x))},\, \, \forall (x,t) \in (0,1) \times (0,\mathcal T),$
\end{itemize}
will be reconstructed and the other data can be obtained by substituting the exact solution in \eqref{s2-1}.\\

In Figure 1, the comparison is presented between the exact fluxes $f_1$ and $f_2$ and their corresponding reconstructed fluxes (obtained after $k^\star$ iterations) at $t=0.5$. This comparison is conducted for various values of $\gamma$ from the set $\{0\%, 0.5\%, 1\%, 5\%\}$. Two cases are considered, corresponding to $\beta$ taking the values $0.3$ and $0.7$. Table \ref{perf_analytic} presents noise levels $\gamma$ and related performance metrics for {\it Example 1}.

\begin{figure}[H]
    \centering
    \begin{subfigure}[t]{0.5\textwidth}
        \centering
        \includegraphics[height=1.8in]{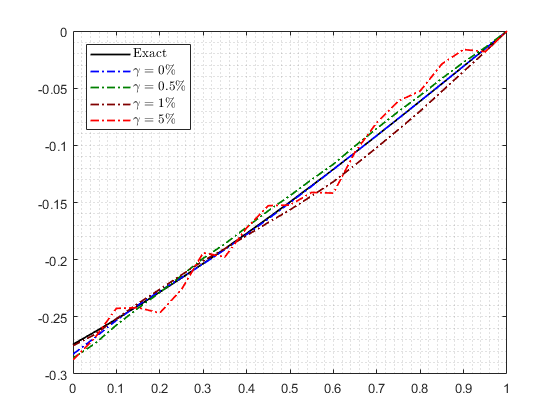}
        \caption{Reconstructed Results for $f_1$}
    \end{subfigure}%
    ~ 
    \begin{subfigure}[t]{0.5\textwidth}
        \centering
        \includegraphics[height=1.8in]{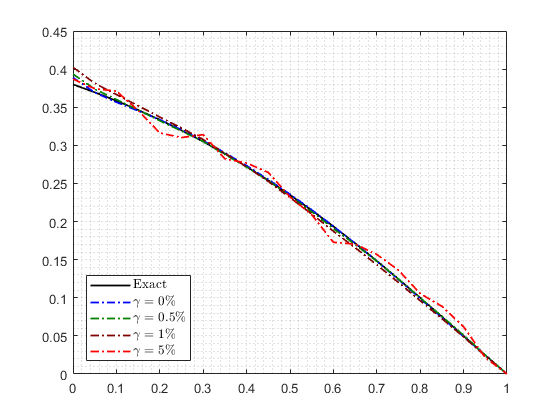}
        \caption{Reconstructed Results for $f_2$}
    \end{subfigure}\\
    \begin{subfigure}[t]{0.5\textwidth}
        \centering
        \includegraphics[height=1.8in]{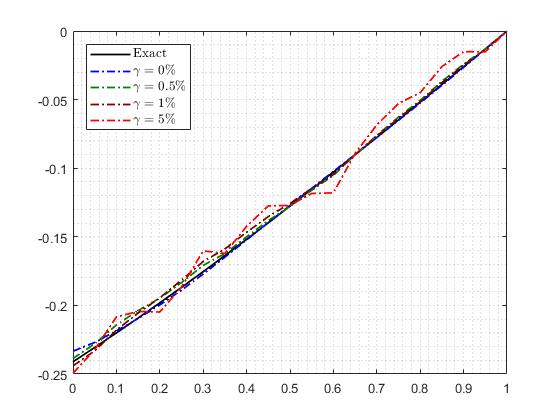}
        \caption{Reconstructed Results for $f_1$}
    \end{subfigure}%
    ~ 
    \begin{subfigure}[t]{0.5\textwidth}
        \centering
        \includegraphics[height=1.8in]{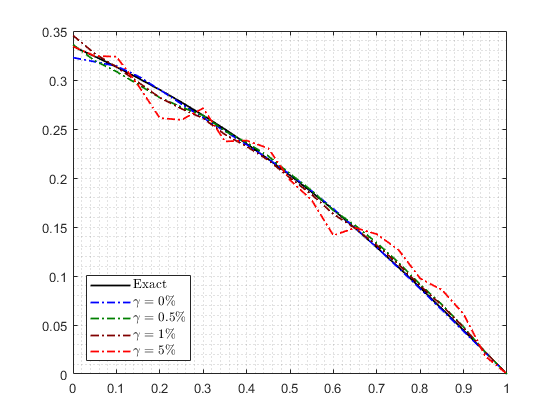}
        \caption{Reconstructed Results for $f_2$}
    \end{subfigure}
    \caption{The numerical results for  {\it Example 1} for various noise levels. The top row corresponds to $\beta=0.3$, and the bottom row corresponds to $\beta=0.7$.}
    \label{exemple1}
\end{figure}
\begin{table}[H]
    \centering
    \small
    \begin{tabular}{|l|c|c|c|c|c|}
        \hline
        & $\gamma$ & 0\% & 0.5\% & 1\% & 5\% \\
        \hline 
        \multirow{4}{*}{\centering $\beta=0.3$} & $\bar{\epsilon}$ & $1.25\times 10^{-7}$ & $1.43\times 10^{-7}$ & $2.86\times 10^{-7}$ & $1.43\times 10^{-5}$ \\
        \cline{2-6}
         & $k^\star$ & 683 & 491 & 151 & 46 \\
        \cline{2-6}
         & $E(k^\star,\gamma,f_1)$ & $4.71e{-03}$ & $1.41e{-02}$ & $1.49e{-02}$ & $1.73e{-02}$ \\
        \cline{2-6}
         & $E(k^\star,\gamma,f_2)$ & $4.87e{-03}$ & $1.92e{-02}$ & $2.01e{-02}$ & $2.14e{-02}$ \\
        \hline
        \multirow{4}{*}{\centering $\beta=0.7$} & $\bar{\epsilon}$ & $1.25\times 10^{-7}$ & $1.47\times 10^{-7}$ & $2.94\times 10^{-7}$ & $1.47\times 10^{-5}$ \\
        \cline{2-6}
         & $k^\star$ & 671 & 485 & 143 & 33 \\
        \cline{2-6}
         & $E(k^\star,\gamma,f_1)$ & $5.01\times 10^{-3}$ & $1.50e{-02}$ & $1.63e{-02}$ & $1.93e{-02}$ \\
        \cline{2-6}
         & $E(k^\star,\gamma,f_2)$ & $4.92e{-03}$ & $1.95e{-02}$ & $2.33e{-02}$ & $2.61e{-02}$ \\
        \hline 
    \end{tabular}
    \caption{Noise levels $\gamma$ with tolerance parameters $\bar{\epsilon}$, iteration steps $k^\star$, and $L_2$ errors ($f_1$ and $f_2$) in {\it Example 1}.}
    \label{perf_analytic}
\end{table}

\paragraph{Example 2.} In this example, the analytic solution of problem \eqref{s2-1} is unknown. The Dirichlet data on the boundaries $\Gamma_1^\mathcal{T}$ and $\Gamma_2^\mathcal{T}$ are constructed by solving the forward problem by using the  iterative approach, developed in subsection \ref{sol-iter}, in which we use the initial data $g(x,y)=0$, and
take a source function $F(x,y,t) = \sin(2\pi x y)$ and the exact fluxes
\begin{itemize}
    \item $f_1(y,t)=e^{-t}\,(t-t^2)\, \sin(3\pi y),\, \, \forall (y,t) \in (0,\ell_y) \times (0,\mathcal T),$
    \item $f_2(x,t)=e^{-t}\,(t-t^2)\, \sin(2\pi x),\, \, \forall (x,t) \in (0,\ell_x) \times (0,\mathcal T).$
\end{itemize}
Figures \ref{exemple1} and \ref{exemple2} illustrate the numerical results obtained for {\it Example 2}, using the conjugate gradient algorithm at the same noise levels chosen in {\it Example 1}. Both cases of $\beta = 0.3$ and $\beta = 0.7$ are included. 
\begin{figure}[H]
    \centering
    \begin{subfigure}[t]{0.5\textwidth}
        \centering
        \includegraphics[height=1.8in]{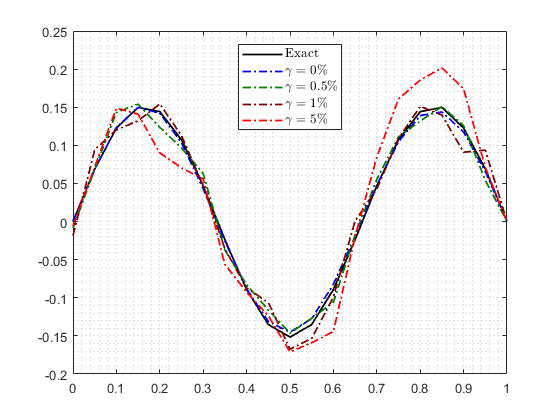}
        \caption{Reconstructed Results for $f_1$}
    \end{subfigure}%
    ~ 
    \begin{subfigure}[t]{0.5\textwidth}
        \centering
        \includegraphics[height=1.8in]{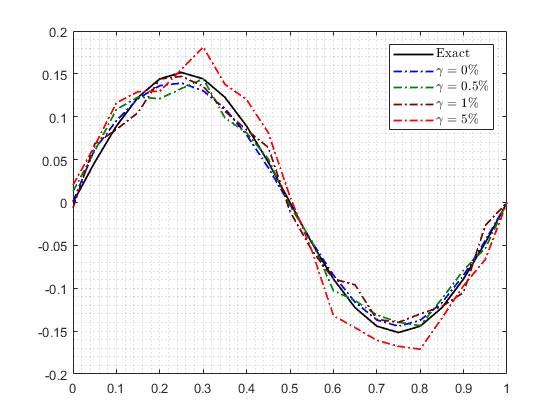}
        \caption{Reconstructed Results for $f_2$}
    \end{subfigure}\\
    \begin{subfigure}[t]{0.5\textwidth}
        \centering
        \includegraphics[height=1.8in]{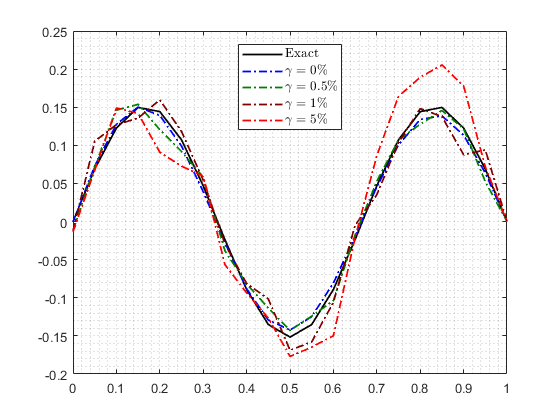}
        \caption{Reconstructed Results for $f_1$}
    \end{subfigure}%
    ~ 
    \begin{subfigure}[t]{0.5\textwidth}
        \centering
        \includegraphics[height=1.8in]{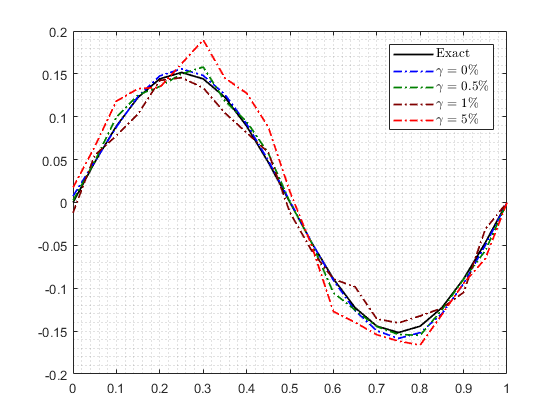}
        \caption{Reconstructed Results for $f_2$}
    \end{subfigure}
    \caption{The numerical results for {\it Example 2} for various noise levels. The top row corresponds to $\beta=0.3$, and the bottom row corresponds to $\beta=0.7$.}
    \label{exemple2}
\end{figure}
The choices of the noise levels $\gamma$ and the corresponding
numerical performances are listed in Table \ref{perf_analytic}.

\begin{table}[H]
    \centering
    \small
    \begin{tabular}{|l|c|c|c|c|c|}
        \hline
        & $\bar{\epsilon}$ & 0\% & 0.5\% & 1\% & 5\% \\
        \hline 
        \multirow{4}{*}{$\beta=0.3$} & $\gamma$ & $1.25\times 10^{-7}$ & $1.85\times 10^{-7}$ & $3.70\times 10^{-7}$ & $1.85\times 10^{-5}$ \\
        \cline{2-6}
         & $k^\star$ & 620 & 472 & 115 & 13 \\
        \cline{2-6}
         & $E(k^\star,\gamma,f_1)$ & $5.31e{-03}$ & $2.11e{-02}$ & $2.31e{-02}$ & $3.10e{-02}$ \\
        \cline{2-6}
         & $E(k^\star,\gamma,f_2)$ & $7.02e{-03}$ & $1.64e{-02}$ & $2.11e{-02}$ & $3.01e{-02}$ \\
        \hline
        \multirow{4}{*}{$\beta=0.7$} & $\gamma$ & $1.25\times 10^{-7}$ & $2.01\times 10^{-7}$ & $4.02\times 10^{-7}$ & $2.01\times 10^{-5}$ \\
        \cline{2-6}
         & $k^\star$ & 661 & 475 & 137 & 16 \\
        \cline{2-6}
         & $E(k^\star,\gamma,f_1)$ & $4.53e{-03}$ & $1.81e{-02}$ & $2.07e{-02}$ & $2.88e{-02}$ \\
        \cline{2-6}
         & $E(k^\star,\gamma,f_2)$ & $4.42e{-03}$ & $1.31e{-02}$ & $1.54e{-02}$ & $2.19e{-02}$ \\
        \hline 
    \end{tabular}
    \caption{Noise levels $\gamma$ with tolerance parameters $\bar{\epsilon}$, iteration steps $k^\star$, and $L_2$ errors in {\it Example 2}.}
    \label{perf_implicit}
\end{table}

The results in Figures \ref{exemple1} and \ref{exemple2}, along with Tables \ref{error_var1} and \ref{error_var2}, showcase remarkable accuracy for both examples. Even with up to $1\%$ noise in the exact Dirichlet data on $\Gamma_1^\mathcal{T}$ and $\Gamma_2^\mathcal{T}$, the obtained results remain highly precise. However, for noise levels exceeding $1\%$ (i.e., $\gamma=5\%$), oscillations emerge, impacting result quality due to the problem's ill-posed nature under higher noise conditions. Additionally, results for $\beta=0.3$ and $\beta=0.7$ are nearly identical, especially when identifying exact fluxes unaffected by the fractional order derivative $\beta$.\\

In the following, we investigate the convergence and stability of our algorithm's performance. Figure \ref{convergencee} demonstrates the convergence results for {\it Example 2}, considering various noise levels $\gamma \in \{0\%, 0.5\%, 1\%, 5\%\}$ while keeping $\beta$ fixed at $0.5$. Figure \ref{func_var} displays the behavior of the objective functional $\tilde{J}(f_1^k, f_2^k)$ concerning the iteration number $k$. This figure provides insights into the convergence behavior of the algorithm when recovering both unknown fluxes, $f_1$ and $f_2$. To assess the algorithm's performance, we examine the variations of the error functions $E(k, \gamma, f_1)$ and $E(k, \gamma, f_2)$ in Figures \ref{error_var1} and \ref{error_var2}, respectively, with respect to the iteration number $k$. These figures allow us to analyze how errors change during the iterative process.
\begin{figure}[H]
  \centering
  \begin{subfigure}[b]{0.3\linewidth}
    \centering
    \includegraphics[height=1.8in]{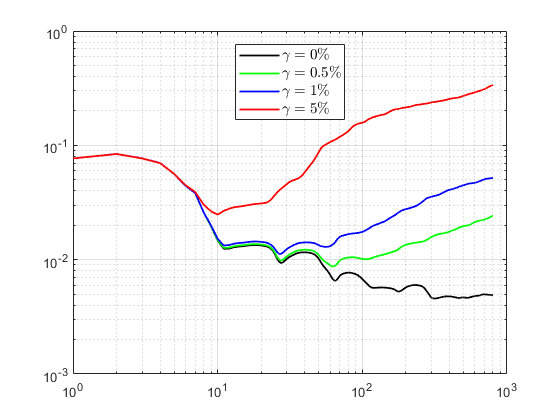}
    \caption{The errors $E(k, \gamma, f_1)$}
    \label{error_var1}
  \end{subfigure}
  \hfill
  \begin{subfigure}[b]{0.3\linewidth}
    \centering
    \includegraphics[height=1.8in]{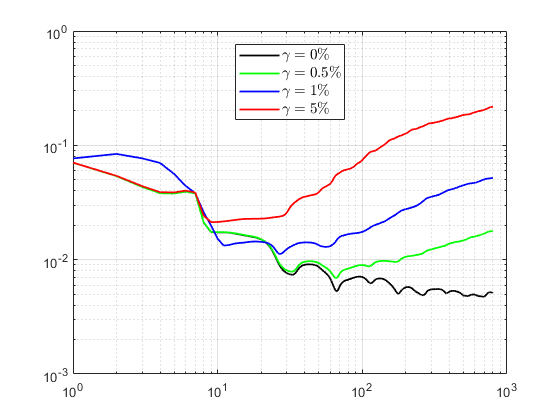}
    \caption{The errors $E(k, \gamma, f_2)$}
    \label{error_var2}
  \end{subfigure}
 \hfill 
  \begin{subfigure}[b]{0.3\linewidth}
    \centering
    \includegraphics[height=1.8in]{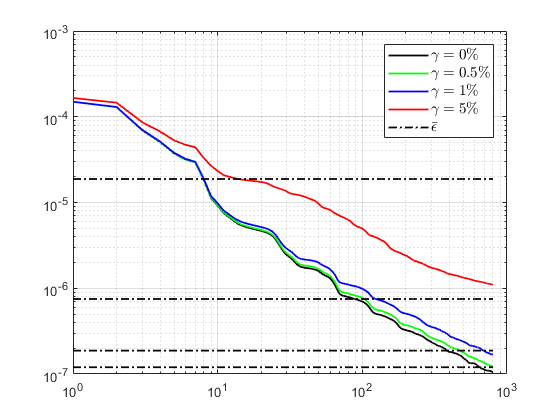}
    \caption{The functional $\tilde{J}(f_1^k, f_2^k)$}
    \label{func_var}
  \end{subfigure}
  \caption{The errors $E(k, \gamma, f_1)$, $E(k, \gamma, f_2)$ and the objective functional $\tilde{J}(f_1^k, f_2^k)$ for {\it Example 2} for various noise levels with $\beta=0.5$.}
  \label{convergencee}
\end{figure}
From Figures \ref{error_var1} and \ref{error_var2}, it's evident that the approximation errors $E(k, \gamma, f_1)$ and $E(k, \gamma, f_2)$ decrease as noise levels decrease. However, after a few iterations, these errors show a slight increase, indicating the need to stop at a suitable step. Figure \ref{func_var} illustrates the objective function's behavior, which consistently decreases with the iteration number $k$ and rapidly converges to a small positive value. This aligns with the prediction mentioned in remark \ref{decroi}. The stopping iteration numbers $k^\star \in \{650, 473, 125, 15\}$ are determined for different values of $\gamma$ ($0\%, 0.5\%, 1\%, 5\%$) based on the discrepancy principle \eqref{stop}. The tolerance parameters $\bar{\epsilon}$ are derived using the expressions $\bar{\epsilon}=\frac{1}{2} \sum_{i=1,2}|h_i^\epsilon-h_i|_{L_2(\Gamma_i^\mathcal{T})}^2$ and \eqref{noiss} for the respective noise levels of $\gamma$ ($0.5\%, 1\%, 5\%$). The $L_2$ errors corresponding to the unknown fluxes $f_1$ and $f_2$ are as follows: $E(k^\star, \gamma, f_1) \in \{\text{4.85e-03}, \text{1.97e-02}, \text{2.12e-02}, \text{2.91e-02}\}$ and $E(k^\star, \gamma, f_2) \in \{\text{4.79e-03}, \text{1.44e-02}, \text{1.62e-02}, \text{2.22e-02}\}$ for the chosen noise levels $\gamma$. These values indicate that the numerical solutions provide reasonably accurate approximations for both fluxes $f_1$ and $f_2$.
\paragraph{Example 3 (Application in Material Science).} In this example, we focus on finding fluxes denoted as $f_1$ and $f_2$ within the context of materials categorized as 'stiff' and 'soft'. To demonstrate the effectiveness of the proposed method, we take the function $k(T^2)$ as in  (\ref{engmaterial}). Our test scenario involves assessing the performance of the proposed iterative method for reconstructing the fluxes $f_1$ and $f_2$ with the following parameters:
\begin{itemize}
    \item {\it Case 1 (soft):} Young’s modulus $E$ is set to 110 GPa, and $T_0^2=0.02.$
    \item {\it Case 2 (stiff):} Young’s modulus $E$ is set to 210 GPa, and $T_0^2=0.027.$
\end{itemize}
In both cases,  the Poisson coefficient  $\nu$ and the  strain hardening exponent $\kappa$ are taken to be 0.3 and 0.5, respectively. Figure \ref{mat_function} illustrates the variation of the plasticity function $k(T^2)$ in both cases (stiff and soft materials). 
\begin{figure}[H]
    \centering
    \includegraphics[height=1.8in]{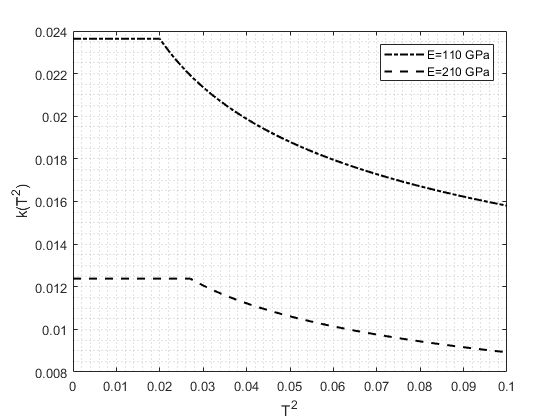}
    \caption{Variation of the plasticity function $k(T^2)$ for stiff and soft engineering materials.}
    \label{mat_function}
\end{figure}
Without loss of generality, we fix the fractional derivative order $\beta=0.5$ and we will present the results of our flux reconstruction efforts at time $t=0.1$ in the two aforementioned cases. Figure \ref{case1} provides comparisons between the exact solutions and the reconstructed ones for \(f_1\) and \(f_2\) in Case 1.
\begin{figure}[H]
    \centering
    \begin{subfigure}[t]{0.5\textwidth}
        \centering
        \includegraphics[height=1.8in]{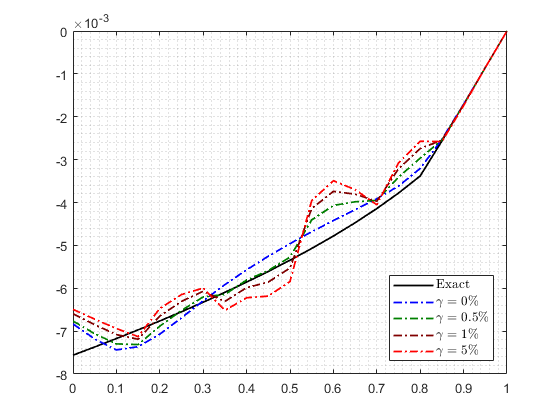}
        \caption{Reconstructed Results for $f_1$}
    \end{subfigure}%
    ~ 
    \begin{subfigure}[t]{0.5\textwidth}
        \centering
        \includegraphics[height=1.8in]{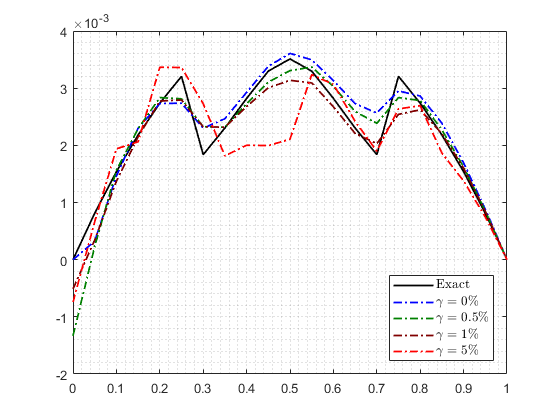}
        \caption{Reconstructed Results for $f_2$}
    \end{subfigure}
    \caption{Comparison of reconstructed results for $f_1$ and $f_2$ in 'soft' materials.}
    \label{case1}
\end{figure}
Figure \ref{case2} illustrates the comparisons between the exact solutions and the reconstructed ones in Case 2.
\begin{figure}[H]
    \centering
    \begin{subfigure}[t]{0.5\textwidth}
        \centering
        \includegraphics[height=1.8in]{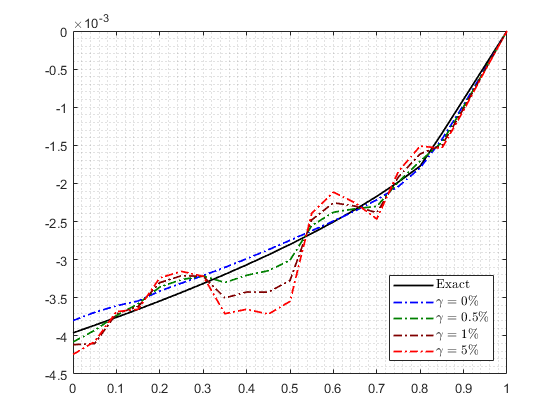}
        \caption{Reconstructed Results for $f_1$}
    \end{subfigure}%
    ~ 
    \begin{subfigure}[t]{0.5\textwidth}
        \centering
        \includegraphics[height=1.8in]{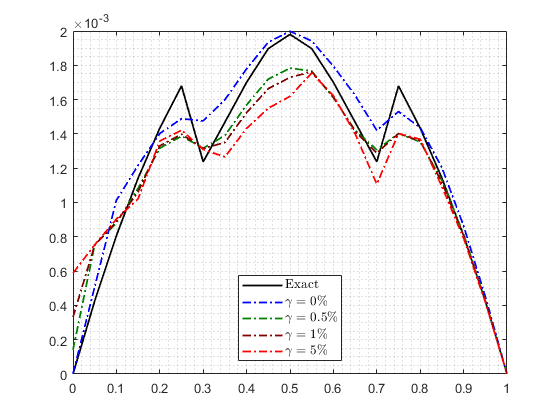}
        \caption{Reconstructed Results for $f_2$}
    \end{subfigure}
    \caption{Comparison of reconstructed results for $f_1$ and $f_2$ in 'stiff' materials.}
    \label{case2}
\end{figure}

In summary, our study successfully demonstrated the effectiveness of our iterative method for flux reconstruction in 'stiff' and 'soft' materials by varying Young's modulus $E$. We consistently achieved accurate and reliable results across different Young's modulus settings, showcasing the robust performance of our approach. Notably, in both cases, the exact solutions to be reconstructed were nonsmooth, affirming that our algorithm excels in the detection of nonsmooth fluxes. These findings highlight its potential applicability in materials science and engineering, particularly in the context of flux detection and characterization.


\begin{thebibliography}{10}

\bibitem{kac}
Kachanov, L. M.
\newblock (1971).  \textit{Fundamentals of the Theory of Plasticity.
\newblock Applied Mathematics and Mechanics 65, North-Holland Publishing Company, Amsterdam-London.}

\bibitem{lan}
Langenbach, A.
\newblock (1971). "Verallgemeinerte und exakte Losungen des Problems der elastischen-plastischen Torsion von Stabben."
\newblock \textit{Mathematische Nachrichten,} 28, 219-234.

\bibitem{Arzu}
A. ~ Hasanov, A. ~ Erdem
\newblock  Determination of unknown coefficient in a non-linear elliptic problem related to the elastoplastic torsion of a bar.
\newblock {\em IMA Journal of Applied Mathematics}, 73, 579 - 591, 2008.

\bibitem{tatar1}
A. ~ Hasanov and S. ~ Tatar.
\newblock Solutions of linear and nonlinear problems related to torsional rigidity of a beam.
\newblock {\em Computational Materials Sciences}, 45, 494-498, 2009.

\bibitem{tatar2}
A. ~ Hasanov and S. ~ Tatar.
\newblock Semi-­analytic inversion method for determination elastoplastic properties of power	hardening materials from limited torsional experiment.
\newblock {\em Inverse Problems in Science and Engineering}, 18, 265-278, 2010.

\bibitem{tatar3}
A. ~ Hasanov and S. ~ Tatar.
\newblock An inversion method for identification of elastoplastic properties of a beam from torsional experiment.
\newblock {\em International journal of Nonlinear Mechanics}, 45, 562-571, 2010.

\bibitem{mam}
A. ~ Mamedov
\newblock An inverse problem related to the determination of elastoplastic properties of a cylindrical bar.
\newblock {\em International journal of Nonlinear Mechanics}, 30, 23-32, 1995.

\bibitem{has}
A. ~ Hasanov
\newblock Inverse coefficient problems for monotone potential operators.
\newblock {\em Inverse Problems}, 13, 1265-1278, 1997.

\bibitem{perona}
P. ~ Perona and J. ~ Malik
\newblock Scale space and edge detection using anisotropic diffusion
\newblock {\em IEEE Trans. Pattern Anal. Mach. Intell.}, 12, 629-639, 1990.


\bibitem{bai}
J. ~ Bai , X.-C. ~ Feng
\newblock Fractional-order anisotropic diffusion for image denoising.
\newblock {\em IEEE Trans. Image Process}, 16, 2492-2502, 207.

\bibitem{wei}
 J. ~ Zhang , Z. ~ Wei
\newblock A class of fractional-order multi-scale variational models and alternating projection algorithm for image denoising.
\newblock {\em Applied Mathematical Modelling}, 35, 2516-2528, 2011.

\bibitem{mako}
Q. ~ Ma , F. Dong, D. ~ Kong
\newblock A fractional differential fidelity-based PDE model for image denoising.
\newblock {\em Machine Vision and Applications}, 28, 635-647, 2017.

\bibitem{osher1}
G. ~ Gilboa , S. ~ Osher
\newblock Nonlocal linear image regularization and supervised segmentation.
\newblock {\em Multiscale Modeling and Simulation}, 6, 595-630, 2007.

\bibitem{osher2}
G. ~ Gilboa , S. ~ Osher
\newblock Nonlocal operators with applications to image processing.
\newblock {\em Multiscale Modeling and Simulation}, 7, 1005-1028, 2008.

\bibitem{gao}
Z. ~ Zhang , Q. ~ Liu, T. Gao
\newblock  A Fast Explicit Diffusion Algorithm of Fractional Order Anisotropic Diffusion for Image Denoising.
\newblock {\em Inverse Problems and Imaging}, 6, 1451-1469, 2021.

\bibitem{xue}
D. L. ~ Chen , S. S. ~ Sun, C. R. ~ Zhang, Y. Q. ~ Chen, D. Y. ~ Xue
\newblock  Fractional order $TV-L^2$ model for image denoising
\newblock {\em Central European Journal of Physics}, 11, 1414-1422, 2013.

\bibitem{pu}
Y. ~ Zhangn , Y. F. ~ Pu, J. R. ~ Hu, J. L. Zhou
\newblock  A class of fractional-order variational image inpainting models.
\newblock {\em Applied Mathematics and Information Sciences}, 6, 299-306, 2012.


\bibitem{tatar10}
Tatar, S.
\newblock (2013b). Numerical solution of the non-linear direct problem related to inverse elastoplastic problem.
\newblock  \textit{Inverse Problems in Science and Engineering,} 21, 52-62.

\bibitem{tatar11}
Tatar, S. and Muradoglu, Z.
\newblock (2014b). Numerical solution of the non-linear evolutional inverse problem related to elastoplastic torsional problem.
\newblock  \textit{Applicable Analysis,} 93, 1187-1200.

\bibitem{men} J. ~ Mendiguren, F. ~Cort\'es and L. ~ Galdos.
\newblock A generalised fractional derivative model to represent elastoplastic behaviour of metals.
\newblock {\em  International Journal of Mechanical Sciences}, 65:12-17, 2012.

\bibitem{bag} R. L. ~ Bagley and P. J. ~ Torvik.
\newblock On the Fractional Calculus Model of Viscoelastic Behavior.
\newblock {\em  Journal of Rheology}, 30 (133), 1986.

\bibitem{pao} M. D. ~ Paola, A. ~ Pirrotta, A. ~ Valenza.
\newblock Visco-elastic behavior through fractional calculus: An easier method for best fitting experimental results.
\newblock {\em  Mechanics of Materials}, 799 - 806, 2011.

\bibitem{tara} V. E. ~ Tarasov, E. C. ~ Aifantis.
\newblock Toward fractional gradient elasticity.
\newblock {\em  J. Mech.  Behav.  Mater.}, 23: 41 - 46, 2014.

\bibitem{sum} W. ~ Sumelka, K. ~ Szajek, T.  ~ Lodygowski,
\newblock Plane strain and plane stress elasticity under fractional continuum mechanics.
\newblock {\em  Archive of Applied Mechanics}, 85: 1527 - 1544, 2015.

\bibitem{meral} F. C ~ Meral, T. J. ~ Royston, R.  ~ Magin,
\newblock Fractional calculus in viscoelasticity: An experimental study.
\newblock {\em  Commun. Nonlinear. Sci. Numer. Simulat.}, 15: 939 - 945, 2010.

\bibitem{su} X. ~ Su, W. ~ Xu, R.  ~ Magin, H. ~ Yang
\newblock Fractional creep and relaxation models of viscoelastic materials via a non-Newtonian time-varying viscosity: physical interpretation.
\newblock {\em Mechanics of Materials.}, 140: 103222, 2020.

\bibitem{sun} H. ~ Sun,  Y. ~ Zhang, D.  ~ Baleanu, W. ~Chen, Y. ~Chen.
\newblock A new collection of real world applications of fractional calculus in science and engineering.
\newblock {\em Mechanics of Materials.}, 213 - 231: 103222, 2018.

\bibitem{srm} S. ~ Tatar,  R. ~ Tinaztepe, M.  ~ Zeki.
\newblock Numerical Solutions of Direct and Inverse Problems for a Time Fractional Viscoelastoplastic Equation.
\newblock {\em Journal of Engineering Mechanics}, 143(7): 2017.

\bibitem{stsuf} S. ~ Tatar,  S. ~ Ulusoy
\newblock Analysis of Direct and Inverse Problems for a Fractional Elastoplasticity Model.
\newblock {\em Filomat}, 31(3): 699 - 708, 2017.

\bibitem{burhan} B. ~ Pektas,  E. ~ Tamci
\newblock  The heat flux identification problem for a nonlinear parabolic equation in 2D.
\newblock {\em Journal of Computational and Applied Mathematics}, 312: 134 - 142, 2017.

\bibitem{CNYY}
J. ~ Cheng, J. ~ Nakagawa, M. ~Yamamoto and T.~ Yamazaki.
\newblock  Uniqueness in an inverse problem for a one-dimensional fractional diffusion equation.
\newblock {\em Inverse Problems}, 25:115--131, 2009.

\bibitem{JR}
B. ~ Jin,  and W.~ Rundell.
\newblock An inverse problem for a one-dimensional time-fractional diffusion problem.
\newblock {\em Inverse Problems}, 28:075010, 2012.

\bibitem{GDXM}
G. ~ Li, D. ~ Zhang,  X.~ Jia and M. ~ Yamamato.
\newblock Simultaneous inversion for the space-dependent diffusion coefficient and the fractional order in the time fractional diffusion equation.
\newblock {\em Inverse Problems}, 29:065014, 2013.

\bibitem{Mis5}
K. ~ Sakamoto and M. ~ Yamamato.
\newblock Inverse source problem with a final overdetermination for a fractional diffusion equation.
\newblock {\em Mathematical Controls and Related Fields}, 4:509-518, 2011.

\bibitem{Mis6}
K. ~ Sakamoto and M. ~ Yamamato.
\newblock Initial value/boundary value problems for fractional diffusion-wave equations and applications to some inverse problems.
\newblock {\em Journal of Mathematical Analysis and Applications}, 382:426-447, 2011.

\bibitem{Mis3}
X ~ Xu, J. ~ Cheng and M. ~ Yamamato.
\newblock Carleman esimate for a fractional diffusion equation with half order and application.
\newblock {\em Applicable Analysis}, 90:1355-1371, 2011.

\bibitem{ZX}
Y. ~ Zhang,  and X.~ Xu.
\newblock Inverse source problem for a fractional diffusion equation.
\newblock {\em Inverse Problems}, 27:035010, 2011.

\bibitem{ry}
R. ~ Lai,  and Y.~ Lin.
\newblock Inverse problems for fractional semilinear elliptic equations.
\newblock {\em Nonlinear Analysis}, 216: 112699, 2022.

\bibitem{ma}
Y. ~ Zhang,  and X.~ Xu.
\newblock An inverse problem study related to a fractional diffusion equation.
\newblock {\em Journal of Mathematical Analysis and Applications}, 512(2): 126145, 2022.


\bibitem{stsu}
S. ~ Tatar, S. ~ Ulusoy.
\newblock An inverse problem for a nonlinear diffusion equation with time-fractional derivative.
\newblock {\em Journal of Inverse and Ill-posed Problems}, vol: 25, No: 2, 185 - 193, 2017.


\bibitem{Kil}
A. A. ~Kilbas, H. M. ~Srivastava and J. J.~ Trujillo.
\newblock Theory and Applications of Fractional Differential Equations.
\newblock {\em Elsevier, Amsterdam}, 2006.


\bibitem{11}
I. ~Podlubny.
\newblock Fractional Differential Equations.
\newblock {\em Academic Press, San Diego}, 1999.

\bibitem{Samko}
S. G. ~Samko, A. A.  ~Kilbas, O. I.  ~Marichev
\newblock Fractional Derivatives and Integrals. Theory and Applications.
\newblock {\em  Gordon and Breach Science Publishers S.A.}, 1993.



\bibitem{Ali}
A. A. ~ Alikhanov.
\newblock A priori estimates for solutions of boundary value problems for fractional order equations.
\newblock {\em Differential Equations}, 46:660--666, 2010.

\bibitem{convex}
I. ~Ekeland, R. Temam.
\newblock Convex Analysis and Variational Problems.
\newblock {\em Stud. Math. Appl., North - Holland Publishing Company,  New York}, 1976.

\bibitem{hasanov1}
A. ~ Hasanov.
\newblock Simultaneous determination of source terms in a linear parabolic problem from the final overdetermination: Weak solution approach.
\newblock {\em J. Math. Anal. Appl.,}, vol: 330, No: 2, 766 - 779, 2007.

\bibitem{salah1}
R. ~ Fletcher and C. M. ~Reeves, 
\newblock Function minimization by conjugate gradients.
\newblock {\em The computer journal}, 7 , 149 - 154, 1964.

\bibitem{lesnic}
K. ~ Cao and D. ~Lesnic, 
\newblock Simultaneous identification and reconstruction of the space-dependent reaction coefficient and source term.
\newblock {\em Journal of Inverse and Ill-posed Problems}, 29(6) , 867 - 894, 2021.

\bibitem{salah2}
Y.- H. ~ Dai and Y.- X. ~ Yuan, 
\newblock Convergence properties of the fletcher-reeves method.
\newblock {\em IMA Journal of Numerical Analysis}, 16 , 155 - 164, 1996.

\bibitem{salah3}
P.  ~ Wolfe, 
\newblock  Convergence conditions for ascent methods.
\newblock {\em SIAM review}, 11 , 226 - 235, 1969.

\bibitem{salah4}
P.  ~ Wolfe, 
\newblock  Convergence conditions for ascent methods. ii: Some corrections.
\newblock {\em SIAM review}, 13 , 185 - 188, 1971.

\bibitem{salah6}
 D. A.  ~ Murio
\newblock  Implicit finite difference approximation for time-fractional diffusion equations.
\newblock {\em  Computers $\&$ Mathematics with Applications}, 56, 1138 - 1145, 2008.


\end{thebibliography}
\end{document}